\documentclass[11pt,a4paper]{article}

\usepackage[T1]{fontenc}
\usepackage{epsfig}
\usepackage{amsmath, amsthm}
\usepackage{amsfonts,amssymb}
\usepackage{derivative}
\usepackage{mathtools}
\usepackage[applemac]{inputenc}		
\usepackage{hyperref}
\usepackage{comment}

\hypersetup{colorlinks=true,urlcolor=black, pdftitle=""}

\newtheorem{conj}{Conjecture}[section]
\newtheorem{thm}[conj]{Theorem}

\newtheorem{rem}[conj]{Remark}
\newtheorem{lem}[conj]{Lemma}
\newtheorem{prop}[conj]{Proposition}

\newtheorem{defn}[conj]{Definition}
\newtheorem{cor}[conj]{Corollary}

\makeatletter
\newtheorem*{rep@theorem}{\rep@title}
\newcommand{\newreptheorem}[2]{%
\newenvironment{rep#1}[1]{%
 \def\rep@title{#2 \ref{##1}}%
 \begin{rep@theorem}}%
 {\end{rep@theorem}}}
\makeatother
\newreptheorem{theorem}{Theorem}
\newreptheorem{corollary}{Corollary}

\newcommand{\vol}{\mathrm{Vol}}

\newcommand{\R}{\mathbb{R}}
\newcommand{\N}{\mathbb{N}}

\usepackage[backend=biber,style=numeric-comp,doi=false, bibencoding=utf8,giveninits=true, maxnames=50, isbn=false,url=false]{biblatex}
\DeclareNameAlias{default}{family-given}
\renewbibmacro{in:}{}

\def\s{\mathbb{S}}

\def\R{{\mathbb R}}

\def\phi{\varphi}

\def\bee{\begin{eqnarray*}}
\def\ene{\end{eqnarray*}}

\newcommand\nnfootnote[1]{%
  \begin{NoHyper}
  \renewcommand\thefootnote{}\footnote{#1}%
  \addtocounter{footnote}{-1}%
  \end{NoHyper}
}

\begin{document}
\title{%
    Weighted Brunn-Minkowski Theory I \\
    \large On Weighted Surface Area Measures}

\author{Matthieu Fradelizi\thanks{Supported in part by the  B\'ezout Labex funded by ANR, reference ANR-10-LABX-58.}, Dylan Langharst\footnotemark[1] 
 \thanks{Supported in part by the U.S.-Israel Binational Science Foundation (BSF) Grant 2018115, and completed  while the authors were in residence at the Institute for Computational and Experimental Research in Mathematics in Providence, RI, during the Harmonic Analysis and Convexity program; this residency was supported by the National Science Foundation under Grant DMS-1929284.}, 
Mokshay Madiman, 
and Artem Zvavitch\footnotemark[1] \footnotemark[2]}

\date{\today}
\maketitle

\begin{abstract}
The Brunn-Minkowski theory in convex geometry concerns, among other things, the volumes, mixed volumes, and surface area measures of convex bodies. We study generalizations of these concepts to Borel measures with density in $\R^n$-- in particular, the weighted versions of mixed volumes (the so-called mixed measures) when dealing with up to three distinct convex bodies. We then formulate and analyze weighted versions of classical surface area measures, and obtain a new integral formula for the mixed measure of three bodies. As an application, we prove a B\'ezout-type inequality for rotational invariant log-concave measures, generalizing a result by Artstein-Avidan, Florentin and Ostrover. The results are new and interesting even for the special case of the standard Gaussian measure.
\end{abstract}
\tableofcontents
\nnfootnote{Keywords: Brunn-Minkowski theory, surface area, Gaussian measure, zonoids, mixed volumes, mixed measures

Mathematics Subject Classification 2020 - Primary: 52A20 and 52A21, Secondary: 46T12
}

\newpage
\section{Introduction}

\subsection{Background}

The study of convex bodies (compact, convex sets in $\R^n$ with non-empty interior) goes back over one hundred years, to the works of Minkowski \cite{Min1903}, Fenchel \cite{Fen36}, and Aleksandrov \cite{Ale38}, among others. One of the core theories in this study is the \textit{Brunn-Minkowski theory}, which focuses on the interaction of the volume of convex bodies and their Minkowski sums. The Brunn-Minkowski theory is thoroughly detailed in the textbook of Schneider \cite{Sch14:book}, and we will make frequent reference to it. For compact sets $K,L\subset\R^n,$ their \textit{Minkowski sum}, or just sum, is precisely $K+L=\{a+b:a\in K,b\in L\}.$ Denoting by $\vol_n$ the Lebesgue measure on $\R^n,$ the \textit{Brunn-Minkowski inequality} states that for $t\in(0,1)$ and compact, convex sets $K$ and $L$
$$\vol_n((1-t)K+t L)^{1/n}\geq (1-t)\vol_n(K)^{1/n}+t\vol_n(L)^{1/n},$$
with equality if, and only if, $K$ and $L$ are homothetic, i.e. $K=aL+b$ for some $a\in \R, b\in \R^n.$ 
One says that $\vol_n$ is $1/n$-concave with respect to Minkowski summation; more generally, a function is $\alpha$-concave for $\alpha>0$ if $f^{\alpha}$ is concave.
Since a weighted arithmetic mean always dominates a weighted geometric mean, one obtains that the volume is also log-concave:
$$\vol_n((1-t)K+t L)\geq \vol_n(K)^{1-t}\vol_n(L)^{t},$$
and, in fact, this is equivalent to the Brunn-Minkowski inequality.

It has been of significant interest to understand convex geometry as being embedded in some more general analytical framework. One such program, dubbed the ``geometrization of probability'' program by V.~Milman, has seen two complementary approaches adopted-- involving log-concave functions (see, e.g., \cite{KM05, Col17}) and log-concave measures (see, e.g., \cite{BM11:cras, MMX17:0}). In another direction, the fact that the Brunn-Minkowski inequality is intimately related to the interaction between the metric and the canonical measure on a Euclidean space has led to vast generalization in the theory of metric measure spaces and their synthetic geometry (see, e.g., \cite{Vil09:book}).


The goal of this work is to explore yet another analytic generalization of Brunn-Minkowski theory-- namely, how other measures on $\R^n$ (specifically, measures with certain concavity properties akin to those possessed by the Lebesgue measure) interact with Minkowski sums of convex bodies, in what we call the \textit{weighted Brunn-Minkowski Theory}. As an example, we will consider \textit{log-concave} measures, where $\mu$ is log-concave if, for any $t\in[0,1]$ and for any compact $K$ and $L$, 
$$\mu((1-t) K +tL)\geq\mu(K)^{1-t}\mu(L)^{t}.$$
The inequality of Pr\'ekopa-Leindler \cite{Pre73, Lei72a, Lei72b} in conjunction with a result of Borell \cite{Bor75a} classifies the log-concave measures by showing that a measure is log-concave if, and only if, its density is a log-concave function on its support. The standard Gaussian measure on $\R^n,$ which is given by $$d\gamma_n(x)=\frac{1}{(2\pi)^{n/2}}e^{-|x|^2/2}dx,$$
where $|\cdot|$ denotes the standard Euclidean norm, is such a measure.

Now that we have given an example of a measure with a concavity property, we discuss some concepts from the Brunn-Minkowski theory and their weighted analogues; several of these will be new. Many of the analogues hold for measures with a concavity property in a very general sense, as explored later in this paper. However, for the sake of readability, we focus on the Gaussian measure in the introduction.

\subsection{Representation formulae for mixed measures}

Let us review some well known facts from convex geometry, which may be found in the textbook of Schneider \cite{Sch14:book}. Steiner's formula states that the volume of the Minkowski sum of two compact, convex sets can be expanded as a polynomial of degree $n$: for every $t \geq 0$, convex body $K$ and compact, convex set $L$, both in $\R^n,$ one has
$$\vol_n(K+tL)=\sum_{j=0}^n \binom{n}{j}t^j V(K[n-j],L[j]),$$
where $V(K[n-j],L[j])$ is the \textit{mixed volume} of $(n-j)$ copies of $K$ and $j$ copies of $L$. When $j=1$, one often writes $V(K[n-1],L)$. By taking the derivative, one obtains
\begin{equation}
	    V(K[n-1],L):=\frac{1}{n}\lim_{\epsilon\to0}\frac{\vol_n(K+\epsilon L)-\vol_n(K)}{\epsilon}=\frac{1}{n}\int_{\s^{n-1}}h_L(u)dS_K(u),
	    \label{eq:mixed_0}
\end{equation}
where $h_L(x)=\sup_{y\in L}\langle y,x\rangle$ is the support function of $L,$ and $S_K$ is the \textit{surface area measure of $K$}. We will discuss the formal definition of the surface area measure below (see Section~\ref{ss:mixed-pty}); essentially, if the Gauss map $n_K:\partial K\to \s^{n-1}$ associates a vector in the boundary $\partial K$ of $K$ with its outer unit normal on the unit sphere ($\s^{n-1}$), then $S_K$ is a Borel measure on the sphere induced by the Gauss map.  Let us mention here that $C^2_+$ convex bodies are those with positive curvature and $C^2$ support function.

The first step in a weighted Brunn-Minkowski theory is to generalize mixed volumes. Since \eqref{eq:mixed_0} has nothing to do with the concavity of the volume, when given an arbitrary Borel measure $\mu,$ and Borel sets $K$ and $L$, the $\mu$-mixed measure of $K$ and $L$ can be defined as
	\begin{equation}
\mu(K;L)=\liminf_{\epsilon\to0}\frac{\mu(K+\epsilon L)-\mu(K)}{\epsilon},
	    \label{eq:arb_mixed_0}
	\end{equation}
when the $\liminf$ is finite. Heuristically, if the limit exists, this is precisely the first coefficient in the Taylor series expansion of $\mu(K+tL)$ (in the variable $t$). This terminology was introduced by Livshyts in \cite{Liv19}, and has been used in other works recently, see e.g. \cite{Liv19,Hos21,LRZ21,KL21:1}. It has appeared previously in many works without being explicitly given the name mixed measures, see e.g. \cite{Naz03,MR14,KM18:1,CENV04,ST74}. For Borel sets $K$ and $L$ containing the origin with finite $\mu$ measure, the limit exists when $\mu$ has continuous density. If $\lambda_n$ denotes the Lebesgue measure, then \eqref{eq:arb_mixed_0} is consistent with mixed volumes up to a factor $n$ i.e. $\lambda_n(K; L)=nV(K[n-1],L)$. 

We emphasize that we deliberately avoid using the notation $\mu(K[n-1], L)$, which some authors have used for $\mu(K;L)$ in the past to allude to the notation for mixed volumes, because the dimension $n$ plays a distinctive role only in the case of the volume (more precisely, the volume of a Minkowski sum of convex sets is a homogeneous polynomial of degree $n$, but this polynomiality generally fails for other measures, where the best we can do is look for coefficients of a putative power series expansion).

One would like to prove an integral representation of mixed measures. The first step, therefore, is to introduce weighted surface area measures. As usual, $\mathcal{H}^{n-1}$ will denote the $(n-1)$-dimensional Hausdorff measure of a given surface. However, for brevity, we may abuse notation to write $dx$ for $d\mathcal{H}^{n-1}(x)$; this is to be understood by context. We will discuss in Section~\ref{ss:mixed-pty} the formal definition of weighted surface area; let us now just explain why it exists. Let $K$ be a convex set; for a Borel measure $\mu$ on $\R^n$ with density $\phi$, denote the $\mu$-measure of $\partial K,$ or the \textit{weighted surface area}, as 
\begin{equation}\mu^+(\partial K):=\liminf_{\epsilon\to 0}\frac{\mu\left(K+\epsilon B_2^n\right)-\mu(K)}{\epsilon}=\int_{\partial K}\phi(y)dy,
\label{eq_bd}
\end{equation}
where the second equality holds if there exists some canonical way to select how $\phi$ behaves on $\partial K$. A large class of functions consistent with \eqref{eq_bd} is when $\phi$ is continuous. That is, if $\mu$ has continuous density $\phi$ on $\R^n$, then it will induce a Borel measure on $\partial K$ which has density $\phi$ with respect to the $(n-1)$-dimensional Hausdorff measure on $\partial K.$ Therefore, one can define the \textit{weighted surface area measure}, denoted $S^{\mu}_{K},$ for any Borel measure $\mu$ on $\R^n$ with continuous density $\phi$ via the Riesz-Markov-Kakutani representation theorem, since, for a continuous $f\in \mathcal{C}(\s^{n-1}),$
$$f\mapsto \int_{\partial K}f(n_K(y))\phi(y)dy$$ 
is a positive linear functional. That is, by definition $S^{\mu}_{K}$ satisfies the following change of variables formula:
$$\int_{\partial K}f(n_K(y))\phi(y)dy=\int_{\s^{n-1}}f(u)dS^{\mu}_{K}(u).
$$

Notationally, one has $S^{\mu}_{K}(\mathbb{S}^{n-1})=\mu^+(\partial K).$ An early work that used $\mu^{+}$ for the weighted surface area is K. Ball's work \cite{Bal93} on Gaussian measure $\gamma_n$. For a compact set $K$, 
$$
\gamma_n^+(\partial K):=\liminf_{\epsilon \to 0} \frac{\gamma_n(K+\epsilon B_2^n)-\gamma_n(K)}{\epsilon}=(2\pi)^{-n/2}\int_{\partial K}e^{-|x|^2/2}dx.
$$
K. Ball showed \cite{Bal93} that for a compact, convex set $K,$ one has $\gamma_n^+(\partial K) \leq 4n^{1/4},$ and F. Nazarov  proved \cite{Naz03} that this bound is asymptotically sharp. Livshyts obtained bounds similar to that of K.~Ball for other classes of rotational invariant log-concave measures \cite{Liv13,Liv14,Liv15,Liv21}.

With this rigorous definition of weighted surface area available, one can prove (see \cite{KL21:2}), using Lemma~\ref{l:second} below,  the following integral representation of mixed measures by setting $f=h_L$ for some compact, convex set $L\subset \R^n$ (see \cite{Liv19} for an alternative proof). 
\begin{prop}[Representation of Mixed Measures]
	Let $L$ be a compact, convex set and $K$ a convex body with the origin in its interior in $\R^n$. Suppose $\mu$ is a Borel measure on $\R^n$ with continuous density. Then, 
	\begin{equation}
	    \mu(K;L)=\int_{\s^{n-1}}h_L(u)dS^{\mu}_{K}(u).
	    \label{eq:arb_mixed}
	\end{equation}
\end{prop}
\noindent In general $\mu(K)\neq \mu(K;K)$, unlike in the volume case (though see the beginning of Section~\ref{ss:mixed-pty} for a different relation between these quantities). 

It turns out there exists also a Steiner formula for  $\vol_n(K+t_1L_1+t_2L_2);$ the coefficient of $t_1t_2$ in the corresponding polynomial expansion is the mixed volume of $(n-2)$-copies of $K,$ one copy of $L_1,$ and one copy $L_2$ and is given by
\begin{equation}
    V(K[n-2],L_1,L_2)=\frac{1}{n}\int_{\s^{n-1}}h_{L_2}(u)dS_{K[n-2],L_1}(u)=\frac{1}{n}\int_{\s^{n-1}}h_{L_1}(u)dS_{K[n-2],L_2}(u),
    \label{eq:mixed_volume}
\end{equation}
where $dS_{K[n-2],L_1}(u)$ and $dS_{K[n-2],L_2}(u)$ are the \textit{mixed area measures}, and the second equality is to emphasize that the mixed volume is invariant under permutation of the compact, convex sets $L_1$ and $L_2$. The surface area measure of the Minkowski sum of compact, convex sets is related to the mixed area measures via a polynomial structure in the same way that volume and mixed volume are related (see \cite[Section 5.1]{Sch14:book}).

Our first (new) contribution to what we call the weighted Brunn-Minkowski theory is the weighted analogue of \eqref{eq:mixed_volume}. 

\begin{defn}
\label{def:mixed}
Let $\mu$ be a Borel measure on $\R^n$. Then, for Borel sets $A,B,C\subset\R^n$ with finite $\mu$-measure, the mixed measure of $(n-2)$ copies of $A$, one copy of $B$ and one copy of $C$ is given by
$$\mu(A;B,C)=\pdv{}{s,t}\mu(A+sB+tC)(0,0)$$
whenever the mixed derivative exists.
\end{defn}

In general $\mu(K)\neq \mu(K;K,K)\neq \mu(K;K)$. However, from the Schwarz theorem of standard multi-variable calculus, one has that $\mu(A;B,C)=\mu(A;C,B)$. For the Lebesgue measure, one has 
$$\lambda_n(K;L,M)=n(n-1)V(K[n-2],L,M).$$ 

We will show below in Remark~\ref{r:mixed_exist} that, if $A$ is a convex body, $B$ and $C$ are compact, convex sets and $\mu$ has a $C^2$ smooth density, then the mixed derivative exists. Furthermore, if $A$ is of the class $C^2_+,$ then we obtain an integral formula for $\mu(A;B,C)$. We state the result for the standard Gaussian measure here.

\begin{thm}
If $A$ is a convex body of the class $C^2_+,$ 
and $B,C$ are compact convex sets in $\R^n$, then 
\begin{equation}\label{eq:mixed_gaussian}
\begin{split}(2\pi)^{n/2}\gamma_n(A;B,C)&=(n-1)\int_{\s^{n-1}}e^{-|\nabla h_{A}(u)|^2/2}h_C(u)dS_{A[n-2],B[1]}(u)
\\
&-\int_{\s^{n-1}}\langle \nabla h_A(u), \nabla h_B(u) \rangle h_C(u)e^{-|\nabla h_{A}(u)|^2/2}dS_A(u).
\end{split}
\end{equation}    
\end{thm}

Our first main result, Theorem~\ref{t:double_deriv}, is a generalization of this formula for more general Borel measures with smooth densities. 

As an example of how this quantity differs from the volume case, let $B=[-\xi,\xi]$ for some $\xi\in \s^{n-1}.$ Then $h_B(u)=|\langle \xi,u \rangle|$ yields $\nabla h_B(u) \overset{\text{a.e.}}{=}\text{sgn}{\langle \xi,u \rangle}\xi$, where $\overset{\text{a.e.}}{=}$ denotes equality up to a set of zero volume and $\text{sgn}(a)=a/|a|$ for $a\in\R\setminus\{0\}$ and $\text{sgn}(0)=0.$ Therefore, we see from \eqref{eq:mixed_gaussian} that
\begin{equation}
    (2\pi)^{n/2}\gamma_n(A;[-\xi,\xi],[-\xi,\xi])=-\int_{\s^{n-1}}\langle \nabla h_A(u), \xi \rangle \langle\xi,u \rangle e^{-|\nabla h_{A}(u)|^2/2}dS_A(u),
\end{equation}
in contrast with the volume case where 
$V(A[n-2],[-\xi,\xi],[-\xi,\xi])=0.$

\subsection{Weighted versions of Minkowski's inequalities}

The Brunn-Minkowski inequality implies various inequalities for the mixed volumes. Minkowski's first and second inequality for volume state that, for compact, convex sets $K$ and $L$ in $\R^n$,
\begin{equation}
    V(K [n-1], L)^{n} \geq \vol_n(K)^{n-1} \vol_n(L),
    \label{eq:first_in}
\end{equation}
and
\begin{equation}
\label{eq:second_in}
V(K[n-1], L)^{2} \geq \vol_n(K) V(K[n-2], L [2]).
\end{equation}
Equality holds in inequality~\eqref{eq:first_in} if and only if $K$ and $L$ are homothetic. We remark that $\vol_n(K)=V(K[n-1],K)=V(K[n-2],K[2])$. Minkowski's second inequality is merely a special case of \textit{Minkowski's quadratic inequality}:  for $A,B,C$ compact, convex sets, one has
$$V(A[n-2],B,C)^2 - V(A[n-2],B,B)V(A[n-2],C,C) \geq 0.$$
This, in turn, is a special case of the Aleksandrov-Fenchel inequality (which we do not discuss). We do mention however that despite having existed in the literature for over a century, the full equality conditions to Minkowski's second and quadratic inequalities were only established very recently by R. van Handel and Y. Shenfeld \cite{SH21}.

In Theorem~\ref{t:first_second}, we generalize the inequalities \eqref{eq:first_in} and \eqref{eq:second_in}. We only discuss the Gaussian case here; the general results can be found in Section~\ref{sec:mixed}. The log-concavity of $\gamma_n$ over the set of Borel sets yields for such $K$ and $L$ that the following Minkowski's first inequality for the Gaussian measure holds:
$$\gamma_n(K;L) - \gamma_n(K;K) \geq \gamma_n(K)\log\left(\frac{\gamma_n(L)}{\gamma_n(K)}\right),$$
with equality if, and only if, $K=L.$ This had been obtained previously in \cite{Liv19}.

However, the Gaussian measure actually has other types of concavity than just log-concavity. Indeed, the Ehrhard inequality states for $0\le t\le 1,$ Borel sets $K$ and $L$ in $\R^{n}$, and the Gaussian measure $\gamma_n$:
\begin{equation}\label{e:Ehrhard_ineq}
\Phi^{-1}\left(\gamma_{n}((1-t) K+tL)\right) \geq(1-t) \Phi^{-1}\left(\gamma_{n}(K)\right)+t \Phi^{-1}\left(\gamma_{n}(L)\right),
\end{equation}
i.e. $\Phi^{-1}\circ\gamma_n$ is concave, where $\Phi(x)=\gamma_{1}((-\infty, x))$. It was first proven by Ehrhard for the case of two closed, convex sets \cite{Ehr83}.  Lata{\l}a \cite{Lat96} generalized Ehrhard's result to the case of an arbitrary Borel set $K$ and a convex set $L$; the general case for two Borel sets of the Ehrhard's inequality was proven by Borell \cite{Bor03}. Since $\Phi$ is log-concave, the the Ehrhard inequality is strictly stronger than the log-concavity of the Gaussian measure, and yields the following analogue of Minkowski's first inequality for Gaussian measure.

\begin{thm}
For Borel sets $K$ and $L$ in $\R^n$, we have 
\begin{equation}\gamma_n(K;L) - \gamma_n(K;K) \geq \sqrt{\frac{1}{2\pi}}e^{-\left(\frac{\Phi^{-1}\left(\gamma_n(K)\right)^2}{2}\right)}\left[\Phi^{-1}(\gamma_n(L))-\Phi^{-1}(\gamma_n(K))\right].
\label{eq:Ehrhard_mixed}
\end{equation}
\end{thm}

The Gaussian measure satisfies other types of concavity if one restricts the sets under consideration. Kolesnikov and Livshyts showed that the Gaussian measure is $\frac{1}{2n}$-concave on the class of convex bodies containing the origin in their interior \cite{KL21:1}. If one further restricts the admissible sets, one can do even better. A compact, convex set $K$ is said to be symmetric if $K=-K.$ Gardner and Zvavitch \cite{GZ10} conjectured that, for symmetric convex bodies $K$ and $L$ and $t\in[0,1]$,
\begin{equation}\label{e:gamma_gaussian}
    \gamma_n\left((1-t) K + t L\right)^{1/n}\geq (1-t)\gamma_n(K)^{1/n} + t \gamma_n(L)^{1/n},
\end{equation}
i.e. $\gamma_n$ is $1/n$-concave over the class of symmetric convex bodies.  An example given in \cite{NT13:1} shows that assumption on $K$ and $L$ having some symmetry is necessary. Important progress was made in \cite{KL21:1}, which led to the proof of the inequality \eqref{e:gamma_gaussian} by Eskenazis and Moschidis in \cite{EM20:3} for symmetric convex bodies. Using this, we obtain the following analogue of Minkowski's first inequality for Gaussian measure when we restrict to symmetric convex bodies.

\begin{thm}\label{thm:gau-symm-mink1}
For symmetric convex bodies $K$ and $L$ in $\R^n$, one has
\begin{equation}
    \label{mixed_gz_gaussian}
    \gamma_n(K;L) - \gamma_n(K;K) \geq n\gamma_n(L)^{1/n}\gamma_n(K)^{\frac{n-1}{n}} - n\gamma_n(K),
\end{equation}
with equality if, and only if, $K=L$.
\end{thm}

Notice that \eqref{mixed_gz_gaussian} is very similar to Minkowski's first inequality \eqref{eq:first_in}. Unfortunately, one cannot improve \eqref{mixed_gz_gaussian} to obtain \eqref{eq:first_in} for the Gaussian measure; we show below that for a convex body $K$ containing the origin in its interior, one always has $\gamma_n(K;K)=n\gamma_n(K) - \int_K|x|^2d\gamma_n(x) < n\gamma_n(K).$  Both \eqref{eq:Ehrhard_mixed} and \eqref{mixed_gz_gaussian} imply an isoperimetric-type result: if $\gamma_n(K)=\gamma_n(L),$ then $\gamma_n(K;L) \geq \gamma_n(K;K).$

Define the following class of Borel measures
\begin{equation}
\label{eq:best_meas}
\begin{split}
    \mathcal{M}:=\left.\{\right.&\mu \text{ Borel measure on }\!\R^n\!:d\mu=e^{-W(|x|)}, 
    \\
    &\left. W:(0,\infty)\to(-\infty,\infty], t\mapsto W(e^t) \text{ is convex}\right\}.
    \end{split}
\end{equation}
This class contains every rotational invariant, log-concave measure as well as the Cauchy measures.
Recently, Cordero-Erasquin and Rotem \cite{CR22} extended the result by Eskenazis and Moschidis to every measure $\mu\in\mathcal{M},$ i.e. every Borel measure $\mu\in\mathcal{M}$ is $1/n$-concave over the same class of symmetric convex bodies. Thus the analogue of Minkowski's first inequality contained in Theorem~\ref{thm:gau-symm-mink1} actually extends to all $\mu\in\mathcal{M}$, which is the content of Theorem~\ref{t:GZCER}.

We, as mentioned, also obtain Minkowski's second inequality for $\mu(A;B,C).$ We present here the case of the Gaussian measure using \eqref{e:gamma_gaussian}; the reader can deduce from the result in Theorem~\ref{t:first_second} other such inequalities for the Gaussian measure and other measures with concavity.

\begin{thm}
    Let $K$ and $L$ be symmetric convex bodies in $\R^n$. Then,
    $$\left(\gamma_n(K;L)-\gamma_n(K;K)\right)^2\geq\frac{n}{n-1} \gamma_n(K)\bigg(\gamma_n(K;L,L)-2\gamma_n(K;K,L)+\gamma_n(K;K,K)\bigg).$$
\end{thm}

\subsection{Bezout-type inequalities and local log-submodularity}

Despite being a very old tool, new facts about the volume of Minkowski sums and mixed volumes are still being discovered. One such area of interest is the study of the reverse of Minkowski's quadratic inequality, in what is known as \textit{B\'ezout-type inequalites}. More precisely, if $\mathcal{C}$ is a class of compact, convex sets closed under Minkowski summation, and given a fixed compact, convex set $A\in\mathcal{C}$, what is the smallest constant $\mathcal{B}_{\mathcal{C}}(A)$ such that every $B,C\in\mathcal{C}$ the following inequality holds
\begin{equation*}
\vol_n(A)V(A[n-2],B,C) \leq \mathcal{B}_{\mathcal{C}}(A)V(A[n-1],B)V(A[n-1],C) ?
\end{equation*}
One then sets $\mathcal{B}_\mathcal{C}=\sup_{A\in\mathcal{C}}\mathcal{B}_\mathcal{C}(A)$. Let $\mathcal{K}^n$ denote the class of all compact, convex subsets of $\R^n$. Then, \textit{Fenchel's inequality} is precisely that $\mathcal{B}_{\mathcal{K}^n} = 2.$ This inequality was first established by Fenchel \cite{Fen36}, and a more accessible proof is in \cite{FGM03}. However, both those proofs use the Aleksandrov-Fenchel inequality. In the sequel to this work \cite{FLMZ23_2}, we establish Fenchel's inequality directly from the Brunn-Minkowski inequality and the limit definition of mixed volumes.

The study of these inequalities started with the B\'ezout inequality for mixed volume, which asserts that $\mathcal{B}_{\mathcal{K}^n}(\triangle_n)=1,$ where $\triangle_n$ is the regular $n$-dimensional simplex. It was conjectured by Soprunov and the last named author \cite{SZ16} that this characterizes the simplex, and this conjecture was confirmed in $\R^2,\R^3$ and when $B,C\in\mathcal{P}^n,$ the class of polytopes in $\R^n$ \cite{SSZ16,SSZ18}. Recent progress has been made towards this conjecture \cite{MS:23_1,MS:23_2}. 

It turns out that showing $\mathcal{B}_{\mathcal{C}}=\frac{n}{n-1}$ is equivalent to volume being log-submodular over that class, which has gained a lot of interest as of late \cite{FMZ22, FMMZ22, BM12:jfa}. This special case is therefore sometimes known as \textit{local log-submodularity} since it follows from differentiation; that is, volume is local log-submodular over a class $\mathcal{C}$ of compact, convex sets if for every $A,B,C\in\mathcal{C}$ one has

\begin{equation}
V(A[n-1],B)V(A[n-1],C)\geq \frac{n-1}{n}\vol_n(A)V(A[n-2],B,C).
\label{eq:AFZ}
\end{equation}

 A \textit{centered zonotope} $Z$ is the Minkowski sum of symmetric line-segments, i.e. it can be written in the form
\begin{equation}
    Z=\sum_{i=1}^ma_i[-u_i,u_i], \; u_i\in\s^{n-1}, \; a_i\in\R.
    \label{eq:zon}
\end{equation}
Furthermore, a \textit{centered zonoid} is the limit, with respect to the Hausdorff metric, of a sequence of centered zonotopes; $\mathcal{Z}^n$ denotes the set of centered zonoids in $\R^n.$ A zonoid (resp. zonotope) is merely a translation of a centered zonoid (resp. zonotope). Due to the translation invariance of the Lebesgue measure, all mentioned results that hold for centered zonoids hold for zonoids; as we will see, the distinction becomes crucial in the weighted case.

A subset of the authors, working with Meyer, showed that $\mathcal{B}_{\mathcal{Z}^n}=\frac{n}{n-1}$ in $\R^2$ and $\R^3$ \cite{FMMZ22}; the $2$-dimensional case follows from \cite{SZ16}. Since every symmetric convex body in $\R^2$ is a zonoid, this means that \eqref{eq:AFZ} holds for all $A,B,C\in\mathcal{K}^2$ \cite{SZ16}. Prior to this work, the case where $A=B_2^n$ and $B,C\in\mathcal{Z}^n$ was established by Hug and Schneider \cite{HS11:1}. Artstein-Avidan, Florentin, and Ostrover \cite{AFO14} extended this result to the case where $A=B_2^n,$ $B\in\mathcal{Z}^n$ and $C\in\mathcal{K}^n$. In fact, they showed the following sharper inequality, with $\kappa_n=\vol_n(B_2^n)$ and $B=Z$:
\begin{equation}
\label{eq:sharp_af}V(B_2^n[n-1],Z)V(B_2^n[n-1],C)\geq \frac{n-1}{n}\frac{\kappa^2_{n-1}}{\kappa_{n-2}\kappa_n}\vol_n(A)V(A[n-2],Z,C),
\end{equation}
which is sharper since $$1\leq \frac{\kappa^2_{n-1}}{\kappa_{n-2}\kappa_n}\leq 1+\frac{1}{n-1}.$$ Unfortunately, since these results hold only for the fixed body $A=B_2^n,$ the equivalence between local log-submodularity and log-submodularity does not hold. If we replace mixed volumes with mixed measures when the measure $\mu$ is the Lebesgue measure $\lambda_n,$ \eqref{eq:sharp_af} becomes
\begin{equation}
    \label{eq:sharp_mixed}
    \lambda_n(B_2^n;Z)\lambda_n(B_2^n;C)\geq \frac{\kappa^2_{n-1}}{\kappa_{n-2}\kappa_n}\vol_n(A)\lambda_n(A;Z,C).
\end{equation}

As an application of our formulas for mixed measures, we extend \eqref{eq:sharp_mixed} to the setting of rotational invariant log-concave measures in Theorem~\ref{t:main} below, and this result reduces directly to \eqref{eq:sharp_mixed} when the measure is set to be the Lebesgue measure. However, an interesting phenomenon occurs; due to the fact that general log-concave measures are not necessarily homogeneous, we replace $B_2^n$ with $RB_2^n$, $R>0,$ and, as a consequence of the proof of Theorem~\ref{t:main}, we can obtain a constant that is monotonically increasing in $R,$ whose minimum value is obtained as $R\to0^+$;  as $R\to0^+,$ the constant reduces to the same constant from the volume case (see \eqref{eq:sharper_constant}). That is, for every $R>0,$ we obtain an inequality sharper than the volume case. However, we can do even better. If one knows a bit about the structure of a given log-concave measure, the constant can be further improved, see Theorem~\ref{t:p}. The Gaussian measure is a special case, see Corollary~\ref{cor:gaf}. That being said, if we know from the beginning we are working with the Gaussian measure and the unit Euclidean ball, we obtain this final result. 
\begin{thm}
\label{t:sharp}
Fix $n\geq 2$. Let $Z$ be a centered zonoid in $\R^n$ and $C$ a compact, convex set in $\R^n$. Then, one has
$$\gamma_n(B_2^n;Z)\gamma_n(B_2^n;C)\geq e^{-\frac{(2n+1)}{2(n+1)^2}}\frac{n}{n-1} \frac{\kappa^2_{n-1}}{\kappa_{n-2}\kappa_n} \gamma_n(B_2^n)\gamma_n(B_2^n;Z,C).$$
\end{thm}

This paper is organized as follows. Section~\ref{sec:mixed} is dedicated to deriving some basic of properties of mixed measures, explored further in the sequel \cite{FLMZ23_2} (in Section~\ref{sec:pmm}), deriving formulas for mixed measures (in Section~\ref{ss:mixed-pty}) and exerting some effort to further establish the theory in the case of the Gaussian measure in the plane (Section~\ref{sec:Gauss_plane}). In Section~\ref{llog_sub}, we establish, as an application of our  formulas, a Gaussian counterpart to a reverse Aleksandrov-Fenchel type inequality originally done in the volume case by Artstein-Avidan, Florentin, and Ostrover \cite{AFO14}. Finally, in Section~\ref{sec:con_re}, we list some concluding remarks concerning connections between the measure of Minkowski sums of compact sets and mixed measures, which are explored in the companion paper \cite{FLMZ23_2}.

\section{Mixed Measures and Weighted Surface Area}
\label{sec:mixed}

\subsection{Properties of Mixed Measures}
\label{sec:pmm}
In this section, we establish some properties for mixed measures.  For a convex set $K$ containing the origin, notice that, for $t\in [0,1]$, \eqref{eq:arb_mixed_0} yields $\mu(tK;K)=\odv{}{t}\mu(tK),$ where the limit exists almost everywhere since the function $\mu(tK)$ is monotonic in $t.$ Consequently, integrating yields
\begin{equation}
    \mu(K)=\int_0^1\mu(tK;K)dt.
    \label{eq:measure_mixed_relate_0}
\end{equation}

We also note, by writing out the limit definition of the derivative, that for a Borel measure $\mu$ on $\R^n$ and Borel sets $A,B$ and $C$ in $\R^n$ such that $\mu(A;B,C)$ from Definition~\ref{def:mixed} exists, one has
\begin{equation}
    \mu(A;B,C)=\lim_{s\to 0}\frac{\mu(A+sB;C)-\mu(A;C)}{s}.
    \label{eq:limit_def_alt}
\end{equation}
In particular, one sees that if $A$ is convex
$$\mu(tA;A,C)=\odv{\mu(tA;C)}{t}.$$
We therefore deduce that
\begin{equation}
    \mu(A;C)=\int_0^1\mu(tA;A,C)dt.
    \label{eq:int_formula_double_mixed}
\end{equation}

\noindent Notice that if $A$ and $B$ are convex, then
\begin{align}
    \mu(A;B,B) &= \pdv{}{s,t}\mu(A+sB+tB)(0,0) \nonumber \\&=\pdv{}{s,t}\mu(A+(s+t)B)(0,0)=\odv[2]{}{s}\mu(A+sB)\big|_{s=0}. \label{eq:2BorNot}
\end{align}

\subsection{Weighted Surface Area and Mixed Measure of two bodies}
\label{ss:mixed-pty}

The formal definition of the surface area measure is the following, for $K\subset \R^n$ a compact, convex set: 
\begin{equation}
    \label{eq:surface}
    S_{K}(E)=\int_{n_K^{-1}(E)}d\mathcal{H}^{n-1}(x)
\end{equation}
for every Borel measure $E \subset \s^{n-1},$ where $\mathcal{H}^{n-1}$ is the $(n-1)$ dimensional Hausdorff measure.
Here, the Gauss map $n_K:\partial K \to \s^{n-1}$ associates an element on the boundary of $K$ with its outer unit normal. The Gauss map is unique for almost all $x\in\partial K.$ A convex body is said to be \textit{strictly convex} if $\partial K$ does not contain a line segment, in which case the Gauss map is an injection between $\partial K\setminus \{x: n_K(x) \text{ is not unique}\}$ and $\s^{n-1}$. The Gauss map is related to the support function: fix $u\in\s^{n-1}$. Then, $\nabla h_K(u)$ exists if, and only if, $n_K^{-1}(u)$ is a single point $x\in\partial K,$ and, furthermore, $\nabla h_K(u)=x=n^{-1}_K(u)$ \cite[Corollary 1.7.3]{Sch14:book}. Hence, $K$ is strictly convex if, and only if, $h_K\in C^1$ \cite[Page 115]{Sch14:book}. 

If a convex body $K$ has positive radii of curvature everywhere, we say $K$ has \textit{positive curvature}. In this instance, there exists a continuous, strictly positive function $f_K(u)$, the \textit{curvature function} of $K$, such that one has $dS_K(u)=f_K(u)du$. It is standard to denote strictly convex bodies with positive curvature and twice differentiable support functions as being of the class $C^2_+.$ From \cite[Theorem 2.7.1]{Sch14:book}, every compact, convex set can be uniformly approximated by convex bodies that are $C^2_+.$ Next, we define the weighted surface area measure of a Borel measure $m$ defined on the boundary of a convex body $K$.

\begin{defn}
\label{def:surface_mu}
	For a compact, convex set $K\subset \R^n$ and a Borel measure $m$ on $\partial K$,  the $m$-surface area of $K$ is the pushforward of $m$ by $n_K:\partial K\to \s^{n-1}$ (i.e., $S^{m}_K=n_{K} \,\sharp\, m$). In the case where $m$ has a density $\phi$, then
\begin{equation}
    \label{eq:surface_mu}
    S^{m}_K(E)=m(n_K^{-1}(E))=\int_{n_K^{-1}(E)}\phi(x)d\mathcal{H}^{n-1}(x)
\end{equation}
for every Borel measurable $E \subset \s^{n-1}.$ If $K$ is $C_+^2$ then $dS^{m}_K(u)=\phi\left(n_K^{-1}(u)\right)f_K(u)du.$
\end{defn}
As discussed in the introduction, given a Borel measure $\mu$ on $\R^n$ with continuous density, there exists a canonical way to select how it behaves on $\partial K.$ Therefore, the measure $S^{\mu}_{K}$ satisfies \eqref{eq:surface_mu} when $m$ is identified with $\mu$. In other words, $S^{\mu}_{K}$ as shown to exist in the introduction is the same as $S^{m}_K$ as defined in  Definition~\ref{def:surface_mu} with $dm=\phi d\mathcal{H}^{n-1}$. 

Recalling that $\mu(tK;K)=\odv{}{t}\mu(tK)$ when $K$ is a convex set, we can use \eqref{eq:arb_mixed} in conjunction with \eqref{eq:measure_mixed_relate_0} when $K$ is a convex body to obtain
\begin{equation}
    \mu(K)=\int_0^1\int_{\s^{n-1}}h_K(u)dS^{\mu}_{tK}(u).
    \label{eq:measure_mixed_relate}
\end{equation}
Note that if $K$ is of class $C^2_+$ with the origin in its interior, then \eqref{eq:arb_mixed} becomes
$$\mu(K;L)=\int_{\s^{n-1}}h_L(u)\phi(n_K^{-1}(u))f_K(u)du,$$
and \eqref{eq:measure_mixed_relate} becomes
\begin{equation}
    \mu(K)=\int_{\s^{n-1}}h_K(u)f_K(u)\int_0^1t^{n-1}\phi\left(tn^{-1}_K(u)\right)dtdu.
    \label{eq:measure_mixed_relate_C2}
\end{equation}
We now show that if a measure $\mu$ has radially decreasing density (where $\varphi$ is said to be radially non-increasing if $\varphi(tx)\geq \varphi(x)$ for every $t\in[0,1]$), then one can relate $\mu(K;K)$ and $\mu(K).$ In the proposition below, $(0,y]$ denotes the line segment from the origin to a vector $y,$ which does not contain the origin.
\begin{prop}
\label{p:meas_relate}
    Let $\mu$ be a Borel measure on $\R^n$ with radially non-increasing density $\varphi.$ Then, for every convex body $K$ in $\R^n$ containing the origin in its interior such that $\varphi$ is defined on $\partial K$, one has
    $$n\mu(K) \geq \mu(K;K),$$
    with equality if, and only if, for almost every $y\in \partial K,$ $\varphi$ is a constant almost everywhere on $(0,y].$
\end{prop}
\begin{proof}
    The result follows from formula \eqref{eq:measure_mixed_relate}, as applying the change of variable formula satisfied by $S^\mu_{tK}$ yields
    $$\mu(K)=\int_0^1\int_{\partial tK}h_K(n_{tK}(y))\varphi(y)dydt=\int_0^1t^{n-1}\int_{\partial K}h_K(n_{K}(y))\varphi(ty)dydt,$$
    where, in the second step, a variable substitution $y\to ty$ was done. Next, Fubini's theorem yields
    $$\mu(K)=\int_{\partial K}h_K(n_{K}(y))\int_0^1t^{n-1}\varphi(ty)dtdy.$$
    The hypothesis that $\varphi$ is radially non-increasing yields $$\int_0^1nt^{n-1}\varphi(ty)dt\ge\varphi(y);$$ this estimate, combined with another use of formula~\eqref{eq:measure_mixed_relate}, completes the proof. Equality occurs if, and only if, $\varphi(ty)$ is a constant for almost every $t\in(0,1]$. Notice that $(0,y]=\{ty:t\in(0,1]\}.$ Thus, equality implies that $\varphi$ is constant almost everywhere on $(0,y]$.
 \end{proof}
To elaborate on the equality conditions of the above proposition, it is possible that, for two different $y_1,y_2\in\partial K,$ $\varphi$ is constant on $(0,y_1]$ and $(0,y_2]$, but the value of $\varphi$ on each segment is different. That is, equality occurs if, and only if, $\varphi$ is the $0$-homogeneous extension of a function on $\partial K$. 

Proposition~\ref{p:meas_relate} implies that $n\gamma_n(K)> \gamma_n(K;K).$ However, we can do better in this case. We will use the notation $\Delta f$ for the Laplacian of a twice-differentiable function $f.$
\begin{prop}
Let $K$ be a convex body containing the origin in its interior. Then, $$\gamma_n(K;K)=n\gamma_n(K)-\int_{K}|x|^2d\gamma_n(x).$$
\end{prop}
\begin{proof}
    From \eqref{eq:arb_mixed} we have that
    $$\gamma_n(K;K)=\int_{\s^{n-1}}h_K(u)dS^{\gamma_n}_K(u)=(2\pi)^{-n/2}\int_{\partial K}h_K(n_K(y))e^{-|y|^2/2}dy,$$
    where the second equality follows from the change of variables formula satisfied by $S^{\gamma_n}_K.$
    However, from the convexity of $K,$ the supremum in the definition of support function will be obtained at $y,$ i.e. $h_K(n_K(y))=\langle n_K(y),y \rangle$. Let $g(y)=-e^{-|y|^2/2}.$ Then, we have
    $$\gamma_n(K;K)=(2\pi)^{-n/2}\int_{\partial K}\langle \nabla g(y),n_K(y) \rangle dy.$$
    From Green's first identity, this is
    $$\gamma_n(K;K)=(2\pi)^{-n/2}\int_{K}\Delta g(x) dx.$$
    But, $$\Delta g(x)=(n-|x|^2)e^{-|x|^2/2},$$
    and so the claim follows.
\end{proof}

Now that we have explored properties of mixed measures, we work towards our first main result, Theorem~\ref{t:double_deriv}. We recall that, for every positive $f\in C(\s^{n-1})$, the \textit{Wulff shape} of $f$ is the convex body given by
	\begin{equation}
	    [f]=\{x\in\R^n:\langle x,u\rangle\leq f(u) \; \forall u \; \in \s^{n-1}\}.
	\end{equation}
	One has, for a convex body $K$ containing the origin in its interior, $[h_K]=K.$ Since $f$ is positive, $[f]$ is such a convex body. Furthermore, if $f$ is even, then $[f]$ is symmetric. In \cite{KL21:2}, the following was shown, expanding on the results from \cite{Liv19, LMNZ17}.

\begin{lem}[Aleksandrov's variational formula for arbitrary measures]
	\label{l:second}
	Let $\mu$ be a Borel measure on $\R^n$ with locally integrable density $\phi$. Let $K$ be a convex body containing the origin in its interior, such that $\partial K$, up to set of $(n-1)$-dimensional Hausdorff measure zero, is in the Lebesgue set of $\phi$. Then, for a continuous function $f$ on $\s^{n-1}$, one has that
	$$\lim_{t\rightarrow 0}\frac{\mu([h_K+tf])-\mu(K)}{t}=\int_{\s^{n-1}}f(u)dS^{\mu}_{K}(u).$$
	\end{lem}	

\begin{rem}
Fix a convex body $K$ containing the origin in its interior, a compact, convex set $L$ and some $\lambda_0>0.$ Notice that $h_K+(\lambda+\lambda_0) h_L=h_{K+\lambda_0 L} +\lambda h_L.$ Hence, an immediate consequence of Lemma~\ref{l:second} is that
$$\lim_{t\rightarrow 0}\frac{\mu(K+(t+t_0)L)-\mu(K+t_0L)}{t}=\int_{\s^{n-1}}h_L(u)dS^{\mu}_{K+t_0 L}(u).$$
Moreover, we can also calculate the variation of the convex combination of $K$ and $L$. For a fixed $\lambda_0\in(0,1),$ write that $(1-\lambda_0)h_K+\lambda_0 h_L=h_K+\lambda_0 (h_L-h_K).$ Then, perturb $\lambda_0$ by a small $\lambda>0$ and write $h_K+(\lambda_0+\lambda) (h_L-h_K)=h_{(1-\lambda_0)K +\lambda_0 L}+\lambda (h_L-h_K)$. Hence, from Lemma~\ref{l:second}, we can conclude that
\begin{equation}
    \label{eq:convex_combo}
    \begin{split}
    \odv{}{\lambda}\mu\left((1-\lambda)K +\lambda L\right)\bigg|_{\lambda_0}&=\odv{}{\lambda}\mu\left(\left[h_{(1-\lambda_0)K +\lambda_0 L}+\lambda (h_L-h_K)\right]\right)\bigg|_{\lambda=0}
    \\
    &=\int_{\s^{n-1}}(h_L-h_K)dS^{\mu}_{K_{\lambda_0}}(u)
    \\
    &=\mu(K_{\lambda_0};L)-\mu(K_{\lambda_0};K),
    \end{split}
\end{equation}
where $K_{\lambda_0}=(1-\lambda_0)K +\lambda_0 L$ and the last equality follows from \eqref{eq:arb_mixed}. We will have occasion to use this observation later.     
\end{rem}

\subsection{Integral representation formula for $\mu(A;B,C)$}

We are now ready to obtain an integral representation for $\mu(A;B,C)$ defined in Definition~\ref{def:mixed}. For our purposes, we need only the case where $A$ is $C^2_+$. First we define the weighted analog of the mixed surface area measure $S_{A[n-2],B[1]}$, which we denote $S^\mu_{A;B}$.

\begin{defn}\label{defn:WMSAM}
Let $A$ be a $C^2_+$ convex body and $B$ be an arbitrary compact, convex set in $\R^n$, $n\geq 2$, and
$\mu$ be a Borel measure with $C^1$ density $\phi$. The  weighted mixed surface area measure $S^\mu_{A;B}$ is the signed measure on $\s^{n-1}$ defined by
$$
dS^\mu_{A;B}(u)=\phi(n^{-1}_A(u))dS_{A[n-2],B[1]}(u) + 
\frac{1}{n-1}\langle\nabla \phi(n^{-1}_A(u)),\nabla h_B(u)\rangle dS_{A}(u) .
$$
\end{defn}

Observe that if we na{\"i}vely define $d\tilde{S}^\mu_{A;B}=\phi(n^{-1}_A(u))dS_{A[n-2],B[1]}(u)$ in analogy with the weighted surface area, then we can write
\begin{equation}\label{eq:WMSAM}
dS^\mu_{A;B}(u)=d\tilde{S}^\mu_{A;B}(u)  + 
\frac{1}{n-1}\bigg\langle\frac{\nabla \phi(n^{-1}_A(u))}{\phi(n^{-1}_A(u))},\nabla h_B(u)\bigg\rangle dS^{\mu}_{A}(u) .    
\end{equation}
Clearly, when $\phi\equiv 1$, the second term vanishes and the weighted mixed surface area measure becomes the (usual) mixed surface area measure.
We emphasize that, in general, $S^\mu_{A;B}$ is only guaranteed to be a signed measure, and may not be a measure.

\begin{thm}
\label{t:double_deriv}
Let $\mu$ be a Borel measure on $\R^n,$ $n\geq 2,$ with $C^2$ density $\phi$. For a $C^2_+$ convex body $A$ and compact, convex sets $B$ and $C$, one has
\begin{equation}
\label{eq:deriv}
\begin{split}
\mu&(A;B,C)=(n-1)\int_{\s^{n-1}}h_C(u) dS^\mu_{A;B}(u).
\end{split}
\end{equation}
\end{thm}

\begin{proof}
We first consider the case when $B$ is a $C^2_+$ convex body. Using \eqref{eq:limit_def_alt}, we compute $$\odv{\mu(A+sB;C)}{s}(0)=\odv{}{s}\left(\int_{\s^{n-1}}\phi(n^{-1}_{A+sB}(u))h_C(u)dS_{A+sB}(u)\right)\bigg|_{s=0},$$ where $n^{-1}_{A+sB}(u)$ is well defined as the Minkowski sum of $C^2_+$ bodies is also $C^2_+$. From \cite[Theorem 5.1.7]{Sch14:book}, we obtain that
\begin{equation}
    S_{A+sB}=S_{A}+s(n-1)S_{A[n-2],B[1]} + O(s^2).
    \label{eq:surface_poly}
\end{equation}
Therefore,  when taking the derivative in $s$ at $s=0$ of  $\phi(n^{-1}_{A+sB}(u))dS_{A+sB}(u),$ only the first two terms in \eqref{eq:surface_poly} contribute.

All that remains is to take derivative of $\phi(n^{-1}_{A+sB}(u)).$ We recall that if $K$ is $C_2^+$, then  $n^{-1}_K(u)=\nabla h_K(u)$ for all $u\in\s^{n-1}$, and thus  $n^{-1}_{A+sB}(u)=\nabla (h_{A+sB}(u))=\nabla h_A(u)+s\nabla h_B(u)$, for all $u\in \s^{n-1}$. Therefore, we obtain for all $u\in\s^{n-1}$ that
\begin{equation}\odv{\phi(n^{-1}_{A+sB}(u))}{s}\bigg|_{s=0}=\langle\nabla \phi(n^{-1}_A(u)),\nabla h_B(u)\rangle.
\label{eq:phi_mixed}
\end{equation}
Thus, we have shown \eqref{eq:deriv} in the case when $B$ is $C^2_+.$ We also notice that, for $s>0,$ by setting $K(s)=A+sB,$ we obtain
\begin{align*}\odv{\mu(A+sB;C)}{s}(s)&=\odv{\mu(K(s)+\tilde{s}B;C)}{\tilde{s}}(0)=(n-1)\!\!\int_{\s^{n-1}}\!\!\!\phi(n^{-1}_{K(s)}(u))h_C(u)dS_{{K(s)}[n-2],B[1]}(u)
\\
&+\int_{\s^{n-1}}\langle\nabla \phi(n^{-1}_{K(s)}(u)),\nabla h_B(u)\rangle h_C(u)dS_{{K(s)}}(u).
\end{align*}
We next consider the general case when $B$ is a compact, convex set. First, approximate $B$ by a sequence of $C^2_+$ convex bodies $\{B_i\},$ such that  $B_i\to B$ uniformly in the Hausdorff metric (see \cite[Theorem 2.7.1]{Sch14:book}; in particular note that this means $h_{B_i}$ converges to $h_B$ uniformly on $\s^{n-1}$).
For $s$ small (say $s \in [0,1]$), and $i\in \N$, let $K_i(s)=A+sB_i$ and consider the function
$$g_i(s)=\mu(K_i(s);C)=\int_{\s^{n-1}}h_C(u)dS^{\mu}_{{K_i}(s)}(u).$$
We show that, for fixed $s,$ $g_i(s)\to g(s)=\mu(K(s);C).$
First, observe that
$$|g_i(s)-g(s)|\leq \int_{\s^{n-1}}|h_C(u)||dS^{\mu}_{K_i(s)}(u)-dS^{\mu}_{K(s)}(u)|.$$
  Since $C$ is a compact, convex set, $|h_C(u)|$ is bounded and consequently it suffices to show that $S^{\mu}_{K_i(s)} \to S^{\mu}_{K(s)}$ weakly. Notice for every Borel $E\subset \s^{n-1},$ one has
  $$|S^{\mu}_{K_i(s)}(E) - S^{\mu}_{K(s)}(E)|=\bigg|\int_{n^{-1}_{K_i(s)}(E)}\phi(x)dx-\int_{n^{-1}_{K(s)}(E)}\phi(x)dx\bigg|.$$
  From \cite[Theorem 4.11]{Sch14:book}, one has $dS_{K_i(s)}\to dS_{K(s)}$ weakly, as $K_i(s)\to K(s)$ in the Hausdorff metric. Therefore, since $\phi$ is bounded on compact sets, the convergence of $S^{\mu}_{K_i(s)}(E)$ to $S^{\mu}_{K(s)}(E)$ follows. Since $E$ was an arbitrary Borel subset of $\s^{n-1},$ we have the weak convergence of $dS^{\mu}_{K_i(s)}(u)$ to $dS^{\mu}_{K(s)}(u),$ therefore the uniform convergence of $g_i(s)$ to $g(s)$.
  
  Next, our goal is to show the uniform convergence of $g'_i(s)$. Notice that
\begin{align*}g^\prime_i(s)&=\mu(K_i(s);B_i,C)=(n-1)\int_{\s^{n-1}}\phi(n^{-1}_{K_i(s)}(u))h_C(u)dS_{{K_i(s)}[n-2],B_i[1]}(u)
\\
&+\int_{\s^{n-1}}\langle\nabla \phi(n^{-1}_{K_i(s)}(u)),\nabla h_{B_i}(u)\rangle h_C(u)dS_{{K_i(s)}}(u).
\end{align*}
Note that $n^{-1}_{K(s)}$ may not be well-defined for $s\neq 0.$ Therefore, we shall show that $\{g_i^\prime(s)\}$ is a Cauchy sequence, to obtain that it has a limiting function $z(s)$. Then, using a standard theorem from classical analysis, we get that $g$ must be differentiable and $g'(s)=z(s)$, in particular, $\mu(A;B,C)=g'(0)=z(0).$ On the other-hand, by computing $z(0)$ via convergence of the integral formula for $g^\prime_i(0)$, we will finish the proof.

Consider $\epsilon >0,$ fix some $N=N(\epsilon)$ (to be determined later), and pick $i,j >N$. Then, it suffices to bound the following five integrals: 
\begin{enumerate}
    \item $\bigg|\int_{\s^{n-1}}\phi(n^{-1}_{K_i(s)}(u))h_C(u)\left(dS_{{K_i(s)}[n-2],B_i[1]}(u)-dS_{{K_j(s)}[n-2],B_j[1]}(u)\right)\bigg|,$
    \item $\int_{\s^{n-1}}|\phi(n^{-1}_{K_i(s)}(u))-\phi(n^{-1}_{K_j(s)}(u))| |h_C(u)|dS_{{K_j(s)}[n-2],B_j[1]}(u),$
    \item $\bigg|\int_{\s^{n-1}}\langle\nabla \phi(n^{-1}_{K_j(s)}(u)),\nabla h_{B_i}(u)\rangle h_C(u)\left(dS_{{K_i(s)}}(u) - dS_{{K_j(s)}}(u)\right)\bigg|,$
    \item $\int_{\s^{n-1}}|\langle\nabla \phi(n^{-1}_{K_j(s)}(u)),\nabla h_{B_i}(u)-\nabla h_{B_j}(u)\rangle| |h_C(u)| dS_{{K_j(s)}}(u),$
    \item $\int_{\s^{n-1}}|\langle\nabla \phi(n^{-1}_{K_i(s)}(u))-\nabla \phi (n^{-1}_{K_j(s)}(u)),\nabla h_{B_i}(u)\rangle| |h_C(u)|dS_{{K_i(s)}}(u).$
\end{enumerate}
Since $h_C$ is bounded on $\s^{n-1}$, we shall not discuss $h_C(u)$ for the rest of the argument.

The easiest integral is the third one. We recall that, since $h_{B_i}$ is convex and $C^2,$ it has bounded derivative on $\s^{n-1}.$  Also, since $\nabla \phi$ is bounded,  all $K_i(s)$ can be taken to belong to some large ball for all $i$ and $s\in [0,1]$, to obtain that $|\langle\nabla \phi(n^{-1}_{K_i(s)}(u)),\nabla h_{B_i}(u)\rangle|$ is bounded and the bound can be taken to be independent of $i$. Since $dS_{K_i(s)}\to dS_{K(s)}$ weakly, $\int_{\s^{n-1}}{dS_{K_i(s)}(u)}$ is a Cauchy sequence. Combining all of this, the third integral is bounded, i.e. we can pick $N$ large enough so the third integral is bounded by $\epsilon/5$.

We next bound the first integral. The argument is exactly the same as the previous one, except that we appeal to the proof of \cite[Theorem 5.1.7]{Sch14:book}, which shows that $dS_{{K_i(s)}[n-2],B_i[1]}$ converges weakly to $dS_{{K(s)}[n-2],B[1]}$  and so
$\int_{\s^{n-1}}dS_{{K_i(s)}[n-2],B_i[1]}(u)$ is also a Cauchy sequence, and, by making $N$ larger if need be, the first integral is also bounded by $\epsilon/5$.

For the fourth integral, we have that $|\nabla \phi(n^{-1}_{K_i(s)}(u))| \leq L_\phi$ for some positive constant $L_\phi.$ Thus, it suffices to show that $\|\nabla h_{B_i}(u)-\nabla h_{B_j}(u)\|_{L_2(\s^{n-1})}$ goes to zero. This follows from the classical analysis fact that if a sequence of convex, differentiable functions converges to some convex function, then their derivatives also converge almost everywhere. Consequently, dominated convergence theorem yields the result, and, thus for $N$ large enough, the fourth integral is bounded by $\epsilon/5.$

For the second integral, we  use  that $\phi$, being $C^1$, is Lipschitz on compact sets, thus, for $s\in [0,1],$ 
\begin{equation}|\phi(n^{-1}_{K_i(s)}(u))-\phi(n^{-1}_{K_j(s)}(u))| \leq L_\phi |n^{-1}_{K_i(s)}(u)-n^{-1}_{K_j(s)}(u)|=L_\phi s|\nabla h_{B_i}(u)-\nabla h_{B_j}(u)|.
\label{eq:lipscthiz_fact}
\end{equation}
Thus, like in the previous argument, the second integral is bounded by $\epsilon/5.$ Finally, for the fifth integral, we use that $\phi$ being $C^2$ yields $\nabla \phi$ is a Lipschitz map on compact sets, and we argue similarly to \eqref{eq:lipscthiz_fact}. Consequently, the fifth integral can also be bounded by $\epsilon/5.$

\noindent Thus, we have shown that for all $s\in [0,1],$ $u\in\s^{n-1},$ and $i,j>N,$ that $|g^\prime_i(s)-g^\prime_j(s)| \leq \epsilon$.
\end{proof}

\begin{rem}
    In the above, all we actually require is that the density of $\mu$ is $C^2$ in a neighborhood of $A.$ Furthermore, in the case when $B$ is of the class $C^2_+,$ we can weaken the condition on the density of $\mu$ to only being Lipschitz in a neighborhood of $A.$
\end{rem}

\begin{rem}
\label{r:mixed_exist}
In fact, we can define the weighted mixed surface area measure even when $A$ is not $C^2_+$. Inspired by Lemma~\ref{l:second}, we first define
$\mu(K;f):=\int_{\s^{n-1}}f(u)dS^{\mu}_{K}(u)$ for any $f\in C(\s^{n-1})$, so that the representation formula for the mixed measure $\mu(K;L)$ may be written concisely as $\mu(K;L)=\mu(K;h_L)$. Now define, for any $f\in C(\s^{n-1})$, the quantity
\begin{equation}
    \mu(A;B,f)=\lim_{s\to 0}\frac{\mu(A+sB;f)-\mu(A;f)}{s} ;
\end{equation}
this definition is, of course, inspired by \eqref{eq:limit_def_alt}. We observe that for any fixed convex body $A$ containing the origin in its interior and compact, convex set $B$, $\mu(A;B,\cdot)$ is a continuous linear functional on $C(\s^{n-1})$, and the Riesz-Markov-Kakutani representation theorem therefore guarantees that it can be written as an integral with respect to a signed measure on the sphere. This signed measure is what we define as $(n-1)dS^\mu_{A;B}$, to be consistent with Theorem~\ref{t:double_deriv}. Note that unlike the weighted surface area, which was defined as the pushforward of a measure on $\partial K$ by the Gauss map in Definition~\ref{def:surface_mu}, we have not given an explicit description for the weighted mixed surface area measure $S^\mu_{A;B}$; this is, however, not surprising since such an explicit description is not available even in the case of Lebesgue measure.
\end{rem}

\begin{rem}
We note that the expression \eqref{eq:WMSAM} for $S^\mu_{A;B}$ is particularly nice in the case of a probability measure $\mu$ on $\R^n$, because in this case, the function 
$\nabla \phi/\phi$ that appears in the second term is the so-called score function that arises naturally in the definition of Fisher information. The score function is central in the study of entropy power inequalities (see, e.g., \cite{Sta59, MB07, JMMR23}).
For example, when $\mu=\gamma_n$ is the Gaussian measure, the score function  
$\nabla \phi(x)/\phi(x)=-x$, so that we have the pleasant form
\begin{equation}\label{eq:gaus-WMSAM}
dS^{\gamma_n}_{A;B}(u)=d\tilde{S}^{\gamma_n}_{A;B}(u)  -
\frac{1}{n-1}\langle n^{-1}_A(u)),n^{-1}_B(u))\rangle dS^{\gamma_n}_{A}(u) ,
\end{equation}
with 
$$
d\tilde{S}^{\gamma_n}_{A;B(u)}=
\frac{1}{(2\pi)^{n/2}} e^{-|n^{-1}_A(u))|^2 /2} dS_{A[n-2],B[1]}(u) ,
$$
when $A$ and $B$ are both $C^2_+$ convex bodies.
\end{rem}

\noindent We note that combining \eqref{eq:2BorNot} with Theorem~\ref{t:double_deriv} yields
$$\odv[2]{}{s}\mu(A+sB)\big|_{s=0}=(n-1)\int_{\s^{n-1}}h_B(u) dS^\mu_{A;B}(u),$$
which is reminiscent of \cite[Proposition 3.2]{KL21:1}.

\subsection{Weighted versions of Minkowski's first and second inequalities}

We conclude this section by establishing Minkowski's first and second inequalities. In the introduction, we discussed the many different types of concavity satisfied by the Gaussian measure. These different types of concavities can be expressed in one definition. A Borel measure $\mu$ on $\R^n$ is said to be $F$-concave on a class $\mathcal{C}$ of Borel subsets of $\R^n$ if there exists a continuous, invertible, monotonic function $F:(0,\mu(\R^n))\to (-\infty,\infty)$ such that, for every pair $K,L \in \mathcal{C}$  and every $\lambda \in [0,1]$, one has $\mu(K),\mu(L) < \infty$ and
  \begin{equation} \mu((1-\lambda)K +\lambda L)\geq F^{-1}\left((1-\lambda)F(\mu(K)) +\lambda F(\mu(L))\right). \label{eq:fcon} \end{equation}
  The case when $F(x)=x^s,$ $s\in \R\setminus\{0\},$ is known as $s$-concavity. Borell's hierarchy completely characterizes $s$-concave Borel measures $\mu$ when the class $\mathcal{C}$ is all compact subsets of $\R^n$ \cite[Theorem 3.2]{Bor75a}. Notice that the $1/n$-concavity of the Guassian measure over symmetric convex bodies is strictly outside of Borell's classification. We emphasis that, if $F$ is increasing, like $x^s,$ $s>0,$ then $F\circ \mu$ is a concave function over $\mathcal{C}.$ Likewise, if $F$ is decreasing, like $x^s, s<0,$ then $F\circ \mu$ is a convex function over $\mathcal{C}.$ Additionally, it is not hard to show that, if there is equality in \eqref{eq:fcon} for a single $\lambda\in (0,1),$ then there is equality for every $\lambda \in (0,1).$ 
  \begin{prop}
	\label{p:equal}
	Let $\mu$ be $F$-concave on a class of Borel sets $\mathcal{C}$. Suppose we have $A,B\in\mathcal{C}$ and some $\lambda\in(0,1)$ such that
	\begin{equation*}
	\mu\left((1-\lambda)K+\lambda L\right) = F^{-1}\left((1-\lambda) F(\mu(K)) +\lambda F(\mu(L))\right)
	\end{equation*}
	Then, equality holds for all $\lambda \in [0,1]$.
	\end{prop}
	\begin{proof}
	Since $F$ is monotone and invertible, it is either increasing or decreasing, and so the function given by $f(t)=F\left(\mu\left((1-t)K+t L\right)\right)$
	is either concave or convex. Denote the linear function $g(t)=(1-t) F(\mu(K)) +t F(\mu(L))$. We have that $g(0)=f(0)=\mu(K)$ and $g(1)=f(1)=\mu(L)$. By hypothesis however, we also have $f(\lambda)=g(\lambda)$. But, since either $f$ or $-f$ is concave, this implies $f=g$ on all of $[0,1].$
	\end{proof}

We now establish Minkowski's first and second inequality for $F$-concave measure of a class of compact, convex sets, using the formulas established. The case of Minkowski's first inequality had been previously established without equality conditions by Livshyts \cite{Liv19}.
\begin{thm}[Minkowski's first and second inequalities for $F$-concave measures and convex bodies]
\label{t:first_second}
Let $\mu$ be a Borel measure on $\R^n$ that is $F$-concave on a class of compact, convex sets $\mathcal{C}.$ Then, for a convex body $K$ containing the origin in its interior, such that $F(x)$ is differentiable at $x=\mu(K)$, and a compact, convex set $L$, both in $\mathcal{C}$, one has Minkowski's first inequality:
$$
\mu(K; L) \geq \mu(K; K)+\frac{F(\mu(L))-F(\mu(K))}{F^{\prime}(\mu(K))},$$
with equality if, and only if, there is equality in \eqref{eq:fcon} for some $\lambda \in (0,1)$. 

Furthermore, if $\mu$ has $C^2$ density and $F(x)$ is twice differentiable at $x=\mu(K),$ then:
\begin{align*}-\frac{F^{\prime\prime}(\mu(K))}{F^{\prime}(\mu(K))}\left(\mu(K;L)-\mu(K;K)\right)^2\geq \left(\mu(K;L,L)
-2\mu(K;K,L)+\mu(K;K,K)\right).
\end{align*}
\end{thm}
\begin{proof}
Consider the function given by $$H(\lambda)=F\left(\mu((1-\lambda) K +\lambda L)\right)- (1-\lambda)F(\mu(K)) -\lambda F(\mu(L)).$$ Thus  $H(0)=H(1)=0.$ Furthermore, if $F$ is an increasing function, then $H$ is concave and $H^\prime(0)\geq 0$. Similarly, if $F$ is a decreasing function, then $H$ is convex and $H^\prime(0)\leq 0$.

Computing the derivative from \eqref{eq:convex_combo} and the chain rule, we obtain directly that
\begin{align}
    H^\prime(0)=F^\prime(\mu(K))(\mu(K;L)-\mu(K;K)) -(F(\mu(L))-F(\mu(K)).
    \label{eq:H_deriv}
\end{align}
Assume $F^\prime(\mu(K))\neq 0.$ If $F$ is increasing, setting this greater than zero and dividing through by $F^\prime(\mu(K))$ yields the first claim. Similarly, if $F$ is decreasing, setting this less than zero and dividing through by the (negative) term $F^\prime(\mu(K))$ yields the same result. Now, suppose $F^\prime(\mu(K))= 0.$ Then, from \eqref{eq:H_deriv}, one obtains in both cases of the monotonicity of $F$ that $\mu(K) \geq \mu(L)$ for all $L\in\mathcal{C}.$ From \eqref{eq:arb_mixed_0}, this yields $\mu(K;L)=\mu(K;K)=0.$ Thus, Minkowski's first inequality is trivial in this case.

Suppose that there is equality in \eqref{eq:fcon}. Then,
$$F\left(\mu((1-\lambda) K+\lambda L)\right)-(1-\lambda) F(\mu(K))-\lambda F(\mu(L))=0, \; \text{for all }\lambda\in[0,1].$$
Differentiating in $\lambda$ yields $$\odv{}{\lambda} F\left(\mu((1-\lambda) K+\lambda L)\right)=F(\mu(L))-F(\mu(K)).$$
Using the chain rule and evaluating at $\lambda=0$ yields
$$
F^{\prime}(\mu(K))\odv{}{\lambda}\mu\left((1-\lambda)K +\lambda L\right)\bigg|_{\lambda=0}=F(\mu(L))-F(\mu(K)).
$$
Inserting \eqref{eq:convex_combo}, we have equality in Minkowski's first inequality. Conversely, suppose that we have equality in Minkowski's first inequality. Then,  running the argument backwards, we get
$$
\odv{}{\lambda} F\left(\mu((1-\lambda) K+\lambda L\right)\bigg|_{\lambda=0}=F(\mu(L))-F(\mu(K)).
$$
This implies equality in \eqref{eq:fcon}. Indeed, let $f(t)=F\left(\mu\left((1-t)K+t L\right)\right)$, which is either concave or convex on $[0,1]$ by hypothesis, and denote the linear function $g(t)=(1-t) F(\mu(K)) +t F(\mu(L))$. We have $f(0)=g(0)$ and $f(1)=g(1)$. However, we have also shown that $f^\prime(0)=g^\prime(0)$. From the concavity of $f$ or $-f$, it follows that $f=g$ on $[0,1]$ via Proposition~\ref{p:equal}.

As for the second inequality, we suppose that $F$ is increasing; the case when $F$ is decreasing is similar. By way of approximation, we first suppose that $K$ is $C^2_+$. We take yet another derivative of $H$ and use that $H^{\prime\prime}(0)\leq 0$ in this instance. One obtains, using \eqref{eq:convex_combo}, that
\begin{equation}F^{\prime\prime}(\mu(K))\left(\mu(K;L)-\mu(K;K)\right)^2+F^\prime(\mu(K))\left(\odv{}{\lambda}\int_{\s^{n-1}}(h_L-h_K)dS^{\mu}_{K_{\lambda}}(u)\right)\bigg|_{\lambda=0}\leq 0,
\label{eq:min_sec_con}
\end{equation}
where $K_\lambda=(1-\lambda)K+\lambda L.$ Next, suppose $L$ is $C_+^2$ and so $K_\lambda$ is also $C^2_+.$ Hence, one has $dS^{\mu}_{K_{\lambda}}(u)=\phi(n^{-1}_{(1-\lambda)K+\lambda L}(u))dS_{K_{\lambda}}(u).$ Now, observe that for almost all $u\in\s^{n-1},$
\begin{equation}\odv{\phi(n^{-1}_{(1-\lambda)K+\lambda L}(u))}{\lambda}\bigg|_{\lambda=0}=\langle\nabla \phi(n^{-1}_{K}(u)),\nabla h_{L}(u)\rangle-\langle\nabla \phi(n^{-1}_{K}(u)),\nabla h_{K}(u)\rangle.
\label{eq:phi_mixed2}
\end{equation}
In order to take the weak derivative of $dS_{K_{\lambda}}(u)$, we expand $S_{K_\lambda}$ as a polynomial in the variable $\lambda$, and obtain from \cite[Theorem 5.1.7]{Sch14:book} that
\begin{equation}
    S_{K_\lambda}=S_{K}+\lambda (n-1)\left(S_{K[n-2],L[1]} - S_{K}\right) + O(\lambda^2).
    \label{eq:week_deriv_sur}
\end{equation}
We deduce that the weak derivative of  $dS_{K_\lambda}(u)$ at $\lambda =0$ is 
$(n-1)\left(dS_{K[n-2],L[1]}(u) - dS_{K}(u)\right).$ Using the product rule on $\phi(n^{-1}_{K_\lambda}(u))dS_{K_\lambda}(u)$ yields the result.
For general compact, convex $L,$ we conclude by the linearity of \eqref{eq:deriv} in the third variable and the approximation argument in the proof of Theorem~\ref{t:double_deriv}. Finally, to remove the assumption that $K$ is $C^2_+$, we appeal to Remark~\ref{r:mixed_exist}, which explains that $\mu(K;K,K),\mu(K;K,L)$ and $\mu(K;L,L)$ still exist in this instance.
\end{proof}

\noindent We list as a special case of Theorem~\ref{t:first_second} the case for measures in $\mathcal{M}$ given by \eqref{eq:best_meas}, where we recall that Cordero-Erasquin and Rotem \cite{CR22} showed every Borel measure $\mu\in\mathcal{M}$ is $1/n$-concave over the class of symmetric convex bodies.
\begin{thm}\label{t:GZCER}
For symmetric convex bodies $K$ and $L$ in $\R^n$ and $\mu\in\mathcal{M}$, one has
\begin{equation*}
    \mu(K;L) - \mu(K;K) \geq n\mu(L)^{1/n}\mu(K)^{\frac{n-1}{n}} - n\mu(K),
\end{equation*}
with equality if, and only if, $K=L$. Additionally, if $\mu$ has $C^2$ density, then
$$\left(\mu(K;L)-\mu(K;K)\right)^2\geq\frac{n}{n-1} \mu(K)\bigg(\mu(K;L,L)-2\mu(K;K,L)+\mu(K;K,K)\bigg).$$
\end{thm}

\subsection{The Gaussian Measure in the Plane}
\label{sec:Gauss_plane}
We conclude this section by exerting some effort to study $\gamma_2.$ Upon observation of \eqref{eq:mixed_gaussian}, it is not clear that $\gamma_n(A;B,C)=\gamma_n(A;C,B)$ without appealing to the Schwarz theorem. We will show explicitly that this is true, and furthermore, use this opportunity to further develop the theory of mixed measures in the plane. We recall that in the plane one has $V_2(A[0],B[1],C[1])=V_2(B,C)$; this does not happen with mixed measures in general, since the inverse Gauss map of $A$ appears in the density of the measure. To emphasis the difference between mixed volume and mixed measures, we will show that mixed measures for $\gamma_2$ can be negative.

Henceforth, we write $u\in\s^{1}$ as $u=u(\theta)=(\cos(\theta),\sin(\theta))$ for $\theta\in [0,2\pi].$ Additionally, we  view the support function of a compact, convex set as a function in $\theta.$ Formally, let $A$ be a convex body and set $h_A(\theta):=h_A(u(\theta)).$ It is beneficial to relate $\nabla h_A (u)$ with $h^\prime_A (\theta).$ Notice that $u(\theta)$ is perpendicular to $u^\prime(\theta)=(-\sin(\theta),\cos(\theta)).$ Then, one has
$$h_A^\prime(\theta)=\langle \nabla h_A(u), u^\prime(\theta) \rangle=-\pdv{h_A(u(x,y))}{x}\sin(\theta)+\pdv{h_A(u(x,y))}{y}\cos(\theta).$$
On the other hand, $h_A(u)$ is $1$-homogeneous. Thus,
$$h_A(u)=\langle \nabla h_A(u),u \rangle=\pdv{h_A(u(x,y))}{x}\cos(\theta)+\pdv{h_A(u(x,y))}{y}\sin(\theta).$$
Writing the vector $\nabla h_A(u)$ in the basis spanned by $u$ and $u^\prime,$ we obtain
\begin{equation}
    \label{eq:form_nabla_h}
    \nabla h_A(u)=\langle \nabla h_A(u),u \rangle u \!+ \langle \nabla h_A(u), u^\prime \rangle u^\prime=h_A(\theta)u(\theta) \!+ h^\prime_A(\theta)u^\prime(\theta); \; u=(\cos(\theta),\sin(\theta)).
\end{equation}
Note, additionally, that, for compact, convex sets $A$ and $B$, one has for almost all $u\in\s^1$
\begin{equation}
    \label{A_B_inner}
    \langle \nabla h_A(u),\nabla h_B(u)\rangle=h_A(\theta)h_B(\theta)+h^\prime_A(\theta)h^\prime_B(\theta).
\end{equation}
We now use the fact that $A$ is $C^2_+$; we observe that our situation simplifies quite a bit in this instance. Firstly, the Monge-Amp{\`e}re equation, which relates the support and curvature functions of a convex body, simplifies to a simple second order, linear, ordinary differential equation in the variable $\theta$ \cite{Gar06:book}:
$$h^{\prime \prime}_A(\theta)+h_A(\theta) = f_A(\theta).$$
Since $A$ is $C^2_+$, we also obtain, for $u=(\cos(\theta),\sin(\theta)),$ $e^{-|n^{-1}_A(u)|^2/2}=e^{-\frac{\left((h_A^\prime(\theta))^2+(h_A(\theta))^2\right)}{2}}.$
Inserting the above two formulae and \eqref{eq:form_nabla_h} into \eqref{eq:mixed_gaussian}, one obtains, in the case where $B$ and $C$ are compact, convex sets and $A$ is a convex body with $C^2_+$ boundary,
\begin{equation}\label{eq:Gaussian_mixed_almost_sym}
    \begin{split}2\pi\gamma_2(A;B,C)&=\int_{0}^{2\pi}e^{-\frac{\left((h_A^\prime(\theta))^2+(h_A(\theta))^2\right)}{2}}h_C(\theta)dS_{B}(\theta)
\\
&-\int_{0}^{2\pi}\left[h_A(\theta)h_B(\theta)+h^\prime_A(\theta)h^\prime_B(\theta)\right] h_C(\theta)e^{-\frac{\left((h_A^\prime(\theta))^2+(h_A(\theta))^2\right)}{2}}dS_{A}(\theta).
\end{split}
\end{equation}
Arguing by approximation, suppose that $B$ is also $C^2_+.$ Then,
$$(h^{\prime \prime}_B(\theta)+h_B(\theta))d\theta = dS_B(\theta).$$
Thus \eqref{eq:Gaussian_mixed_almost_sym} becomes
\begin{align*}2\pi\gamma_2(A;B,C)&=\int_{0}^{2\pi}e^{-\frac{\left((h_A^\prime(\theta))^2+(h_A(\theta))^2\right)}{2}}h_C(\theta)[h^{\prime\prime}_B(\theta)+h_B(\theta)(1-h^2_A(\theta)-h^{\prime\prime}_A(\theta)h_A(\theta))]d\theta
\\
&-\int_{0}^{2\pi}h^\prime_B(\theta)h_C(\theta)h^\prime_A(\theta)e^{-\frac{\left((h_A^\prime(\theta))^2+(h_A(\theta))^2\right)}{2}}[h^{\prime \prime}_A(\theta)+h_A(\theta)]d\theta.
\end{align*}
Next, observe that
$$d\left(e^{-\frac{\left((h_A^\prime(\theta))^2+(h_A(\theta))^2\right)}{2}}\right)=-e^{-\frac{\left((h_A^\prime(\theta))^2+(h_A(\theta))^2\right)}{2}}h^\prime_A(\theta)[h^{\prime \prime}_A(\theta)+h_A(\theta)]d\theta.$$
Consequently, integration by parts yields, with $f_A(\theta)=h^{\prime\prime}_{A}(\theta)+h_A(\theta),$
\begin{equation}\gamma_2(A;B,C)=
\int_{0}^{2\pi}e^{-\frac{\left((h_A^\prime(\theta))^2+(h_A(\theta))^2\right)}{2}}[h_B(\theta)h_C(\theta)(1-h_A(\theta)f_A(\theta))-h^{\prime}_B(\theta)h_C^\prime(\theta)]\frac{d\theta}{2\pi}.
\label{eq:Gaussian_mixed_sym}
\end{equation}
Since \eqref{eq:mixed_gaussian} is independent of the second derivative of the support function of $B,$ the assumption that $B$ is $C^2_+$ in \eqref{eq:Gaussian_mixed_sym} can be dropped via an approximation argument. If $A$ is a ball of radius $R,$ then $h_A=f_A=R.$ Therefore, a remarkable consequence of \eqref{eq:Gaussian_mixed_sym} is the following: given compact, convex sets $B,C\subset \R^2,$ there exists $R$ such that $$\gamma_2(RB_2^n;B,C) <0.$$

\section{Local log-submodularity}
\label{llog_sub}

\subsection{Inequalities for rotationally invariant log-concave measures}

We next generalize inequality \eqref{eq:sharp_mixed} to the case of rotational invariant log-concave measures. We remind the reader that $\mu$ is a rotational invariant log-concave measure if there exists a non-decreasing, convex function $W:\R^+\to \R\cup \{+\infty\}$ such that $d\mu=e^{-W(|x|)}dx.$ The goal of this section is to find a quantity $\mathcal{A}$ such that, for every rotational invariant $\mu$, one has 

\begin{equation}
\mu(A;B)\mu(A;C)\geq \mathcal{A} \frac{\kappa^2_{n-1}}{\kappa_{n-2}\kappa_n}\mu(A)\mu(A;B,C),
\label{eq:objective}
\end{equation}
where we consider the case when $B$ is replaced in the above by some centered zonoid $Z,$ $C$ is an arbitrary compact, convex set, and $A=RB_2^n,$ $R>0$. In the case of volume (that is, \eqref{eq:sharp_mixed}, where $\mathcal{A}=1$), this is the same as considering $A=B_2^n$; both sides of the inequality in \eqref{eq:sharp_mixed} being homogeneous of degree $2n-2$ in the variable $A$.  As we will see, this is not true for \eqref{eq:objective}. 
\begin{prop}
\label{p:rotation_log_ineq}
Let $\mu$ be a Borel measure on $\R^n$ such that $d\mu(x)=e^{-W(|x|)}dx$ for some convex function $W:[0,\infty)\to\R$ that is differentiable at $R>0$. Then, for $A=RB_2^n,$ and compact, convex sets $B,C\subset\R^n,$ \eqref{eq:objective} is equivalent to
\begin{equation}
\label{eq:rotation_log_ineq}
    \begin{split}
        V(B_2^n[n-1],B)&V(B_2^n[n-1],C)\geq \mathcal{A}\frac{n-1}{n}\frac{\kappa^2_{n-1}}{\kappa_{n-2}\kappa_n}I_\mu(R)
        \\
        &\left(V(B_2^n[n-2],B,C)-\frac{RW^\prime(R)}{n(n-1)}\int_{\s^{n-1}}\langle u,\nabla h_B(u)\rangle h_C(u)du\right)
    \end{split}
\end{equation}
where
$$I_\mu(R)=\int_{B_2^n}e^{W(R)-W(R|x|)}dx.$$
\end{prop}

\begin{proof}
Observe that, for a compact, convex set $B\subset \R^n,$ one has
$$\mu(RB_2^n;B)= e^{-W(R)}R^{n-1}\int_{\s^{n-1}}h_B(u)du= e^{-W(R)}R^{n-1}nV(B_2^n [n-1],B).$$
We next compute:
\begin{align*}
&\mu(RB_2^n;B,C)
\\
&= e^{-W(R)} \left(R^{n-2}(n-1)\!\int_{\s^{n-1}}\!\!h_C(u)dS_{B_2^n[n-2],B}(u)- R^{n-1}W^\prime(R)\!\int_{\s^{n-1}}\!\!\langle u,\nabla h_B(u)\rangle h_C(u)du\right)\\
&=R^{n-2}e^{-W(R)}n(n-1)\!\left(\!V(B_2^n[n-2],B,C)-\frac{RW^\prime(R)}{n(n-1)}\int_{\s^{n-1}}\!\!\langle u,\nabla h_B(u)\rangle h_C(u)du\right).
\end{align*}
where, in the last line, we used that $\int_{\s^{n-1}}h_C(u)dS_{B_2^n[n-2],B}(u)=nV(B_2^n[n-2],B,C).$ Next, we write $\mu(RB_2^n)$ in the following way:
\begin{align*}
    \mu(RB_2^n)= \int_{RB_2^n}e^{-W(|x|)}dx=R^n e^{-W(R)}\int_{B_2^n}e^{W(R)-W(R|x|)}dx.
\end{align*}
Inserting each term into \eqref{eq:objective} yields the result.
\end{proof}

Proposition~\ref{p:rotation_log_ineq} shows that $\mathcal{A}=1$ corresponds to the known case of volume. We  now substitute a centered zonoid $Z$ in place of $B$ in Proposition~\ref{p:rotation_log_ineq} and we additionally assume that $W$ is increasing (i.e. the measure $\mu$ is log-concave). We first need a preliminary lemma.
\begin{lem}
\label{l:hom}
Let $f\in L^1(\R^n,\gamma_n)$ be $k$-homogeneous, for some $k>-n$. Then, one has
$$\int_{\R^n}f(x)d\gamma_n(x)=\frac{2^\frac{k-2}{2}\Gamma(\frac{n+k}{2})}{\pi^\frac{n}{2}}\int_{\s^{n-1}}f(u)du.$$
\end{lem}
\begin{proof}
Integrating in polar coordinates, we directly compute:
\[
\int_{\R^n}f(x)d\gamma_n(x)=\int_{0}^{+\infty}r^{n+k-1}e^{-\frac{r^2}{2}}\frac{dr}{(2\pi)^\frac{n}{2}}\int_{\s^{n-1}}f(u)du=
\frac{2^\frac{k-2}{2}\Gamma(\frac{n+k}{2})}{\pi^\frac{n}{2}}\int_{\s^{n-1}}f(u)du.
\]
\end{proof}

\begin{lem}[Proposition~\ref{p:rotation_log_ineq} for centered zonoids]
\label{l:zon}
Let $\mu$ be a Borel measure on $\R^n$ such that $d\mu(x)=e^{-W(|x|)}dx$ for some increasing, convex function $W:[0,\infty)\to\R$ that is differentiable at $R>0$. Then, for $A=RB_2^n,$ $B\in\mathcal{Z}^n$ and a compact, convex set $C\subset\R^n,$ \eqref{eq:objective} holds if $\mathcal{A}$ satisfies
$$\mathcal{A}n\int_0^1e^{W(R)-W(Rt)}t^{n-1}dt\left(1-\frac{RW^\prime(R)}{n}\right)\le 1.$$
\end{lem}
\begin{proof}
Notice that in \eqref{eq:rotation_log_ineq}  we may assume that $B=Z$ is a centered zonotope, and the general case of centered zonoids follows by approximation. Notice also that the expression is dilation invariant in $Z$. Therefore, using the additivity of both mixed volumes and the support function, we may assume that $Z:=[-\theta,\theta]$ for some $\theta\in \s^{n-1}$. One has $h_Z(u)=|\langle  u,\theta\rangle|$. Furthermore (see for example \cite{Sch14:book}), for compact, convex sets $K_2,\dots,K_n$ one has
$$V_n([-\theta,\theta],K_2,\dots,K_n)=\frac{2}{n}V_{n-1}(P_{\theta^\perp}K_2,\dots,P_{\theta^\perp}K_n),$$
where the subscripts emphasis that the first mixed volume is $n$-dimensional, the second mixed volume is $n-1$ dimensional, and the notation $P_{\theta^\perp} K_i$ denotes the orthogonal projection of $K_i$ onto the hyperplane through the origin orthogonal to the direction $\theta$. Therefore, we have
\begin{equation}V(B_2^n[n-2],Z,C)=\frac{2}{n}V_{n-1}(B_2^{n-1}[n-2],P_{\theta^\perp}C)=\frac{2}{n(n-1)}\int_{\s^{n-1}\cap\theta^\perp}h_C(u^\prime)du^\prime.
\label{eq:gwmv}
\end{equation}
Notice that $\nabla h_Z(u)=\theta\chi_{\langle  u,\theta\rangle>0}-\theta\chi_{\langle  u,\theta\rangle<0};$ inserting this, \eqref{eq:gwmv}, $V(B_2^n[n-1],Z)=\frac{2\kappa_{n-1}}{n},$ and $V(B_2^n[n-1],C)$ into \eqref{eq:rotation_log_ineq}, our inequality becomes
\begin{equation}\label{eq:local-zon-gamma}
        \int_{\s^{n-1}}\!\!h_C(u)du\geq \mathcal{A}\frac{\kappa_{n-1}}{\kappa_{n-2}\kappa_n}I_\mu(R)
        \left[\int_{\s^{n-1}\cap\theta^\perp}\!\!h_C(u^\prime)du^\prime-\frac{RW^\prime(R)}{2}\int_{\s^{n-1}}\!\!|\langle u,\theta\rangle| h_C(u)du\right]\!.
\end{equation}
Now, arguing like in \cite{AFO14}, let $\Pi_\theta$ denote the reflection operator with respect to the hyperplane $\theta^\perp$. Explicitly, $\Pi_{\theta}(x)=x-2\theta\langle x,\theta\rangle$. Then, consider the symmetrization $S_{\theta}C$ of $C$ in the direction $\theta$ to be the compact, convex set given by $S_{\theta}C:=\frac{C+\Pi_\theta C}{2}.$ Notice that $P_{\theta^\perp} C\subset S_{\theta}C.$ Furthermore, $\int_{\s^{n-1}}h_C(u)du=\int_{\s^{n-1}}h_{S_{\theta}C}(u)du\geq \int_{\s^{n-1}}h_{P_{\theta^\perp} C}(u)du.$ Hence, the quantity
$$ \int_{\s^{n-1}}h_C(u)du+\frac{\mathcal{A}RW^\prime(R)}{2}\frac{\kappa_{n-1}}{\kappa_{n-2}\kappa_n}I_\mu(R)\int_{\s^{n-1}}|\langle u,\theta\rangle| h_C(u)du$$ is bounded from below by the same quantity, but with $C$ replaced by $P_{\theta^\perp} C$. Consequently, it suffices to consider \eqref{eq:local-zon-gamma} only for the case when $C\subset \theta^\bot$.

It is known, (see, for example,  \cite[page 404, eq. (A.29)]{Gar06:book}), that if $K$ is a $k<n$ dimensional compact, convex set in $\R^n$, then, for $0\le i \le k$, one has
$$\frac{1}{c_{i, k}} V_k\left(K[i], B_{2}^{k}[k-i]\right)=\frac{1}{c_{i, n}} V_n\left(K[i], B_{2}^{n}[n-i]\right), \text { where } c_{i, k}=\frac{\kappa_{k-i}}{\left(\begin{array}{c}
k \\
i
\end{array}\right)}.$$
Using this, if $C\subset \theta^\bot$ then
\begin{align*}
\int_{\s^{n-1}} h_{C}(u) du &= nV_n\left(B_{2}^{n}[n-1],C\right)=n\frac{c_{1, n}}{c_{1, n-1}} V_{n-1}\left(B_{2}^{n-1}[n-2],C\right)
\\
&=\frac{\kappa_{n-1}}{\kappa_{n-2}}\int_{\s^{n-1} \cap \theta^{\perp}} h_{C}(u^\prime) d u^{\prime}.
\end{align*}

We now focus on the quantity $\int_{\s^{n-1}}|\langle u,\theta\rangle| h_C(u)du$. Due to the fact that $C\subset \theta^\bot$, it simplifies greatly. Apply Lemma~\ref{l:hom} to the $2$-homogeneous function $|\langle u,\theta\rangle| h_C(u)$ to obtain
\[
\int_{\s^{n-1}}|\langle u,\theta\rangle| h_C(u)du= \kappa_n\int_{\R^n}|\langle x,\theta\rangle| h_C(x)d\gamma_n(x).
\]
Using Fubini's theorem and rotation invariance to compute this last integral, we get
\[
\int_{\R^n}|\langle x,\theta\rangle| h_C(x)d\gamma_n(x)=
\int_{\R}|x_1|d\gamma_1(x_1)\int_{\theta^\bot}h_C(x)d\gamma_{n-1}(x)=
\frac{2}{\sqrt{2\pi}}\int_{\theta^\bot} h_C(x)d\gamma_{n-1}(x).
\]
Now apply Lemma~\ref{l:hom} to $h_C(u)$, which is $1$-homogeneous, but defined in  dimension $n-1$,  
\[
\int_{\theta^\bot} h_C(x)d\gamma_{n-1}(x)=\frac{2^\frac{-1}{2}\Gamma(\frac{n}{2})}{\pi^\frac{n-1}{2}}\int_{\s^{n-1}\cap\theta^\bot}h_C(u')du'.
\]
We thus obtain
\[
\int_{\s^{n-1}}|\langle u,\theta\rangle| h_C(u)du= \frac{\kappa_n\Gamma(\frac{n}{2})}{\pi^\frac{n}{2}}\int_{\s^{n-1}\cap\theta^\bot}h_C(u')du'=\frac{2}{n}\int_{\s^{n-1}\cap\theta^\bot}h_C(u')du'.
\]
When inserting the above relations for $\int_{\s^{n-1}}|\langle u,\theta\rangle| h_C(u)du$ and $\int_{\s^{n-1}} h_{C}(u) du$ into \eqref{eq:local-zon-gamma}, we see that every term has $\int_{\s^{n-1}\cap\theta^\bot}h_C(u')du'$ in it. It then cancels, and so it suffices to show
$$1\geq \frac{\mathcal{A}}{\kappa_n}I_\mu(R)\left(1-\frac{RW^\prime(R)}{n}\right).$$
Writing $I_\mu(R)$ in polar coordinates, we obtain that
$\frac{1}{\kappa_n}I_\mu(R)=n\int_0^1e^{W(R)-W(Rt)}t^{n-1}dt.$ Inserting this into the above yields the claim.
\end{proof}

We now state the main result of this section, which is that $\mathcal{A}=1$ when $B$ is a centered zonoid and $\mu$ is rotational invariant, log-concave Borel measure.
\begin{thm}
\label{t:main}
Let $\mu$ be a rotational invariant measure on $\R^n$ of the form $d\mu=e^{-W(|x|)}dx$ for an increasing, convex $W:[0,\infty)\to\R$. Then, for every $R>0$, a centered zonoid $Z$ and a compact, convex set $C$, one has
$$\mu(RB_2^n;Z)\mu(RB_2^n;C)\geq  \frac{\kappa^2_{n-1}}{\kappa_{n-2}\kappa_n}\mu(RB_2^n)\mu(RB_2^n;Z,C).$$
\end{thm}

\begin{proof}
Since the function $W$ is convex, it is differentiable almost everywhere; by approximation, we can assume $W$ is differentiable. Then, from Lemma~\ref{l:zon}, it suffices to show that
$$1\geq n\int_0^1e^{W(R)-W(Rt)}t^{n-1}dt\left(1-\frac{RW^\prime(R)}{n}\right).$$
If $R$ is so that $RW^\prime(R)\geq n$ then we are done, as the right-hand-side is non-positive. So, suppose $RW^\prime(R)<n$. Notice, from convexity, one has
$$\frac{W(R)-W(Rt)}{R-Rt}\leq W^\prime(R).$$
Therefore, we see that
$$\left(1-\frac{RW^\prime(R)}{n}\right)n\int_0^1e^{W(R)-W(Rt)}t^{n-1}dt\leq \left(1-\frac{RW^\prime(R)}{n}\right)n\int_0^1e^{W^\prime(R)R(1-t)}t^{n-1}dt.$$
Next, using Taylor series expansion of the exponential function, we use uniform convergence to interchange the integral and the summation:
$$\left(\!1-\frac{RW^\prime(R)}{n}\!\right)n\sum_{k=0}^\infty \frac{(RW^\prime(R))^k}{k!}\int_{0}^1\!(1-u)^ku^{n-1}du\!= \left(\!1-\frac{RW^\prime(R)}{n}\!\right)\sum_{k=0}^\infty (RW^\prime(R))^k\frac{n!}{(k+n)!},
$$
where we computed that $n\int_{0}^1(1-u)^ku^{n-1}du=nB\left(k+1,n\right)=\frac{k!n!}{\left(k+n\right)!},$
with $B(x,y)$ the Beta function. Next,
$$\sum_{k=0}^\infty (RW^\prime(R))^k\frac{n!}{(k+n)!}=1+\sum_{k=1}^\infty \left(\frac{RW^\prime(R)}{1+n}\right)\cdots\left(\frac{RW^\prime(R)}{k+n}\right).$$
But, notice that, since $W^\prime(R)R<n<n+1,$
$$1+\sum_{k=1}^\infty \left(\frac{RW^\prime(R)}{1+n}\right)\cdots\left(\frac{RW^\prime(R)}{k+n}\right)\le \sum_{k=0}^\infty \left(\frac{W^\prime(R)R}{n+1}\right)^k\!=\frac{1}{1-\frac{W^\prime(R)R}{n+1}}=\frac{n+1}{n+1-W^\prime(R)R}.$$
Consequently,
$$\left(1-\frac{RW^\prime(R)}{n}\right)\frac{n+1}{n+1-W^\prime(R)R}=\frac{n+1}{n}\frac{n-RW^\prime(R)}{n+1-W^\prime(R)R}=\frac{n+1}{n}\left(1-\frac{1}{n+1-W^\prime(R)R}\right).$$
But $W^\prime(R)R$ increases with $R$, and thus the quantity above decreases with $R$. Therefore, the quantity is maximized when $R=0$, which is precisely 1.
\end{proof}
Let us show a general approach to obtain sharper constants. Suppose one considered finding $\mathcal{A}_{\mu,R}$ such that, for a Borel measure $\mu$ on $\R^n$ of the form $d\mu(x)=e^{-W(|x|)}dx$ for $W:[0,\infty)\mapsto\R$, and a fixed $R>0$, one has for every compact, convex set $C$ and $Z\in\mathcal{Z}^n$ that
$$\mu(RB_2^n;Z)\mu(RB_2^n;C)\geq \mathcal{A}_{\mu,R} \frac{\kappa^2_{n-1}}{\kappa_{n-2}\kappa_n}\mu(RB_2^n)\mu(RB_2^n;Z,C).$$
One sees from Lemma~\ref{l:zon} that, if $d\mu(x)=e^{-W(|x|)}dx$ with $W:[0,\infty)\to\R$ differentiable, increasing,  and convex, then one needs
\begin{equation}
1\geq \mathcal{A}_{\mu,R}\int_0^1e^{W(R)-W(Rt)}t^{n-1}dt\left(n-RW^\prime(R)\right).
\label{eq:log_integral_bound}
\end{equation}
If $n\leq RW^\prime (R)$, then \eqref{eq:log_integral_bound} is non-positive and so any choice of $\mathcal{A}_{\mu,R}$ works (this implies that $\mu(RB_2^n;Z,C) \leq 0$ in this case, so the bound is trivial). If $n > RW^\prime (R),$ then, from the end of the proof of Theorem~\ref{t:main}, we are able to set $A_{\mu,R}$ to be the reciprocal of the final bound, that is
\begin{equation}A_{\mu,R}=\frac{n}{n+1}\left(1+\frac{1}{n-W^\prime(R)R}\right).
\label{eq:sharper_constant}
\end{equation}
\noindent Notice \eqref{eq:sharper_constant} is increasing in $R$; where the value of $1$ corresponds to $\lim_{R\to 0} \mathcal{A}_{\mu,R}=1.$ So, for an arbitrary $R>0$, $A_{\mu,R}\geq 1$, and equals $1$ if, and only if, $\mu$ is a constant multiple of the Lebesgue measure.

\subsection{Improved inequalities for a special class of measures}

We see, however, that we can do even better by avoiding the estimate $\frac{W(R)-W(Rt)}{R-Rt}\leq W^\prime(R)$ in Theorem~\ref{t:main}. While we cannot avoid this for general rotational invariant $\mu$, we avoid this in the Gaussian case, and also, in general, in the case of measures of the form $d\mu=\alpha e^{-|x|^p/\beta}dx$ for some $\alpha,\beta>0$ and $p\geq 1$. We first start with a technical lemma.
\begin{lem}
\label{l:gaussian_integral}
Let $\beta>0,$ $p,n\geq 1$ and $R>0$ so that $R^p\in\left(0,\beta+\frac{\beta n}{p}\right)$. Then,
$$1+\frac{pR^p}{{p\beta+\beta n}-pR^p}\ge n\int_{0}^1e^{\frac{R^p}{\beta}(1-r^p)}r^{n-1}dr.$$
In particular, for $R=1$ and $p=\beta=2,$
$$1+\frac{1}{n+1}\ge n\int_{0}^1e^{\frac{(1-r^2)}{2}}r^{n-1}dr.$$
\end{lem}
\begin{proof}
Start by writing the right-hand side as
$$\frac{n}{p}\int_{0}^1e^{\frac{R^p}{\beta}(1-r^p)}pr^{n-1}dr.$$
Next, let $u=r^p$. Then, $u^{\frac{n}{p}-1}du=pr^{n-1}dr$, and the right-hand side becomes
$$\frac{n}{p}\int_{0}^1e^{\frac{R^p}{\beta}(1-u)}u^{\frac{n}{p}-1}du.$$
Using the Taylor series expansion of the exponential function, and then the uniform convergence to interchange the integral and the summation, yields:
$$\frac{n}{p}\sum_{k=0}^\infty \frac{R^{kp}}{\beta^k k!}\int_{0}^1(1-u)^ku^{\frac{n}{p}-1}du.$$
Notice that $$\frac{n}{p}\int_{0}^1(1-u)^ku^{\frac{n}{p}-1}du=\frac{n}{p}B\left(k+1,\frac{n}{p}\right)=\frac{k!\Gamma\left(\frac{n}{p}+1\right)}{\Gamma\left(k+1+\frac{n}{p}\right)},$$
where $\Gamma(x)$ is the Gamma function. Inserting this computation into the above series yields
$$\sum_{k=0}^\infty \left(\frac{R^p}{\beta}\right)^k\frac{\Gamma\left(1+\frac{n}{p}\right)}{\Gamma\left(k+1+\frac{n}{p}\right)}.$$
But this can be written as
$$1+\sum_{k=1}^\infty \left(\frac{R^p}{\beta}\right)^k\frac{1}{\left(k+\frac{n}{p}\right)\cdots\left(1+\frac{n}{p}\right)}=1+\sum_{k=1}^\infty \frac{pR^p}{p\beta k+\beta n}\cdots\frac{pR^p}{p\beta+\beta n}.$$
But, notice that, from our choice of $R,$
$$1+\sum_{k=1}^\infty \frac{pR^p}{p\beta k+\beta n}\cdots\frac{pR^p}{p\beta+\beta n}\le \sum_{k=0}^\infty \left(\frac{pR^p}{p\beta+\beta n}\right)^k=\frac{1}{1-\frac{pR^p}{p\beta+\beta n}},$$
which yields our result.
\end{proof}

\begin{thm}
\label{t:p}
Fix $n\geq 2,$ $p\geq 1,$ and $\alpha,\beta>0$. Let $\nu$ be the Borel measure on $\R^n$ given by $d\nu=\alpha e^{-\frac{|x|^p}{\beta}}dx$. Fix $R>0.$ For a centered zonoid $Z$ and a compact, convex set $C,$

$$\nu(RB_2^n;Z)\nu(RB_2^n;C)\geq \mathcal{A}_{\nu,R} \frac{\kappa^2_{n-1}}{\kappa_{n-2}\kappa_n}\nu(A)\nu(RB_2^n;Z,C),$$
where
$$\mathcal{A}_{\nu,R}=\left(1-\frac{pR^p}{p\beta+\beta n}\right)\left(1+\frac{pR^p}{\beta n -pR^p}\right).$$
\end{thm}
\begin{proof}
From Lemma \ref{l:zon} one needs to show that $\mathcal{A}_{\nu,R}$ satisfies 
$$1\geq \mathcal{A}_{\nu,R}\int_{0}^1e^{\frac{R^p}{\beta}(1-r^p)}r^{n-1}dr\left(n-\frac{pR^p}{\beta}\right).$$
If $R^p\geq \frac{\beta n}{p}$ then the right-hand side is non-positive, and we are done. Otherwise, one obtains the result from Lemma~\ref{l:gaussian_integral}, which we can use since $R^p \leq  \frac{\beta n}{p} < \beta +  \frac{\beta n}{p}.$
\end{proof}

We next apply Theorem~\ref{t:p} to the case of the Gaussian measure and the unit ball, to demonstrate straight explicitly how the bounds we obtain improve the volume case and \eqref{eq:sharper_constant}.
\begin{cor}
\label{cor:gaf}
Fix $n\geq 2$. Let $Z$ be a centered zonoid and $C$ a compact, convex set. Then,
\begin{equation}
    \label{eq:gaf}
\gamma_n(B_2^n;Z)\gamma_n(B_2^n;C)\geq\frac{n}{n-1}\frac{n+1}{n+2} \frac{\kappa^2_{n-1}}{\kappa_{n-2}\kappa_n} \gamma_n(B_2^n)\gamma_n(B_2^n;Z,C).
\end{equation}
\end{cor}

From using a geometric series to approximate the Taylor series in Lemma~\ref{l:gaussian_integral}, the above result is still not sharp. If we allow $\beta$ to depend on $p$, then we can improve our bound by using Jensen's inequality. For simplicity, we consider only the case when $R=1.$ For general $R$, the quantities $\beta$ and $p$ would depend on $R$ as well.
\begin{lem}\label{l:gaussian_jensen_integral}
Suppose $p,n\geq 1$ and pick $\beta$ so that $\beta \geq 1+\frac{1}{p-1}$. Then,
$$n\int_0^1 e^{\frac{(1-r^p)}{\beta}}r^{n-1}dr \leq e^{\frac{1-\left(\frac{n}{n+1}\right)^p}{\beta}}.$$
\end{lem}
\begin{proof}
Observe that $nr^{n-1}dr$ is a probability measure on $[0,1]$, Furthermore, $\odv[2]{}{r}e^{\frac{(1-r^p)}{\beta}}=\frac{p}{\beta}e^{\frac{(1-r^p)}{\beta}}r^{p-2}\left[\frac{p}{\beta}r^p-(p-1)\right]$. From the choice of $\beta$, this is less than $0$ or all $r\in[0,1]$. Hence, from Jensen's inequality
$$n\int_0^1 e^{\frac{(1-r^p)}{\beta}}r^{n-1}dr\leq \exp{\left(\frac{1}{\beta}\left[1-\left(n\int_0^1r^{n}dr\right)^p\right]\right)},$$
which is our claim.
\end{proof}

\begin{cor}
\label{cor:equality}
Fix $n\geq 2,$ $p\geq 1,$ $\alpha>0,$ and $\beta\ge 1+\frac{1}{p-1}$. Let $\nu$  be the Borel measure on $\R^n$ given by $d\nu=\alpha e^{-\frac{|x|^p}{\beta}}dx$. For $Z$ a centered zonoid and $C$ a compact, convex set,

$$\nu(B_2^n;Z)\nu(B_2^n;C)\geq \mathcal{A}_{\nu} \frac{\kappa^2_{n-1}}{\kappa_{n-2}\kappa_n}\nu(A)\nu(B_2^n;Z,C),$$
where
$$\mathcal{A}_{\nu}=\frac{n}{n-1}e^{\frac{\left(\frac{n}{n+1}\right)^p-1}{\beta}}.$$
\end{cor}
It is not true a priori that the constant in Corollary~\ref{cor:equality} is sharper than in Theorem~\ref{t:p}. This is true in the case of the Gaussian measure.
\begin{reptheorem}{t:sharp}
Fix $n\geq 2$. Let $Z$ be a centered zonoid in $\R^n$ and $C$ a compact, convex set in $\R^n$. Then, 
$$\gamma_n(B_2^n;Z)\gamma_n(B_2^n;C)\geq e^{-\frac{(2n+1)}{2(n+1)^2}}\frac{n}{n-1} \frac{\kappa^2_{n-1}}{\kappa_{n-2}\kappa_n} \gamma_n(B_2^n)\gamma_n(B_2^n;Z,C).$$
Furthermore, this inequality is sharper than in Corollary~\ref{cor:gaf}.
\end{reptheorem}

\begin{proof}
The first claim follows from Lemma~\ref{l:gaussian_jensen_integral}, Lemma~\ref{l:zon} and the fact that
$$e^{\frac{1-\left(\frac{n}{n+1}\right)^2}{2}}=e^{\frac{2n+1}{2(n+1)^2}}.$$
Now, we must show that
$$e^{\frac{2n+1}{2(n+1)^2}}\leq \frac{n+2}{n+1}=1+\frac{1}{n+1},$$
which is true, as the function given by 
$1+\frac{1}{x+1}-e^{\frac{2x+1}{2(x+1)^2}}$ is positive for all $x\in\R^+$ and decreases to $0$ asymptotically as $x\to \infty$.
\end{proof}

\section{Concluding Remarks}
\label{sec:con_re}

There are questions related to those discussed in this paper that we briefly comment on. Firstly, it was noticed by 
\cite{FMMZ18} that the supermodularity property 
\begin{equation}\label{eq:supmod}
\vol_n(A)+\vol_n(A+B+C)\geq \vol_n(A+B)+\vol_n(A+C)    
\end{equation}
holds for any convex bodies $A, B, C$ in $\R^n$ (in fact, they also conjectured that it should hold for arbitrary compact sets $B, C$ when $A$ is convex, and proved this in dimension 1). It turns out that the possible negativity of the mixed measure $\gamma_n(A;B,C)$ immediately implies (by using an equivalence theorem discussed, for example, in \cite{FMZ22}, where the inequality \eqref{eq:supmod} was shown to be equivalent to nonnegativity of certain mixed volumes) that such a supermodularity property cannot hold when volume is replaced by Gaussian measure. Further details may be found in the companion paper \cite{FLMZ23_2}.

In another direction, although $\vol_n^{(1/n)}$ is not fractionally superadditive on the set of compact sets in $\R^n$ thanks to the counterexample of \cite{FMMZ16} (and hence is neither supermodular nor Schur-concave, as observed in \cite{MNT20}), it was proved recently by \cite{BM22} that $\vol_n$ is fractionally superadditive on compact sets in $\R^n$. Such questions for more general measures than volume are also discussed in the companion paper \cite{FLMZ23_2}.

\printbibliography

\noindent Matthieu Fradelizi
\\
Univ Gustave Eiffel, Univ Paris Est Creteil, CNRS, LAMA UMR8050, F-77447 Marne-la-Vall\'ee, France.
\\
E-mail address: matthieu.fradelizi@univ-eiffel.fr
\vspace{2mm}
\\
\noindent Dylan Langharst
\\
Department of Mathematical Sciences, Kent State University, Kent, OH 44242, USA. 
\\
E-mail address: dlanghar@kent.edu
\vspace{2mm}
\\
\noindent Mokshay Madiman 
\\
University of Delaware, Department of Mathematical Sciences, 501 Ewing Hall, Newark, DE 19716, USA. 
\\
E-mail address: madiman@udel.edu
\vspace{2mm}
\\
\noindent Artem Zvavitch
\\
Department of Mathematical Sciences, Kent State University, Kent, OH 44242, USA. 
\\
E-mail address: azvavitc@kent.edu

\end{document}